\documentclass[10pt,a4paper]{article}

\usepackage{amsmath, amsthm, amssymb, amsfonts}
\usepackage[T1]{fontenc}

\numberwithin{equation}{section}
\newtheorem{theorem}{Theorem}[section]
\newtheorem{lemma}[theorem]{Lemma}
\newtheorem{definition}[theorem]{Definition}
\newtheorem{corollary}[theorem]{Corollary}
\newtheorem{proposition}[theorem]{Proposition}
\newtheorem{remark}[theorem]{Remark}
\newtheorem{example}[theorem]{Example}

\newtheorem{claim}[theorem]{Claim}

\DeclareSymbolFont{AMSb}{U}{msb}{m}{n}
\DeclareMathSymbol{\N}{\mathalpha}{AMSb}{"4E}
\DeclareMathSymbol{\R}{\mathalpha}{AMSb}{"52}
\DeclareMathSymbol{\Z}{\mathalpha}{AMSb}{"5A}
\DeclareMathSymbol{\C}{\mathalpha}{AMSb}{"43}
\DeclareMathSymbol{\p}{\mathalpha}{AMSb}{"50}
\DeclareMathSymbol{\D}{\mathalpha}{AMSb}{"44}
\DeclareMathSymbol{\X}{\mathalpha}{AMSb}{"58}
\newcommand{\M}{\mathsf{M}}
\newcommand{\m}{\mathsf{m}}

\begin{document}

\title{Localization and Tensorization Properties of the Curvature-Dimension Condition for Metric Measure Spaces}

\author{Kathrin Bacher, Karl-Theodor Sturm}

\date{}

\maketitle

\begin{abstract}
This paper is devoted to the analysis of metric measure spaces satisfying
locally the curvature-dimension condition $\mathsf{CD}(K,N)$ introduced by the second author and also studied by Lott \& Villani.
We prove that the
local version of $\mathsf{CD}(K,N)$ is equivalent to a  global condition $\mathsf{CD}^*(K,N)$, slightly weaker than the (usual, global) curvature-dimension condition.
This so-called reduced curvature-dimension condition $\mathsf{CD}^*(K,N)$ has the {\em local-to-global property}.
We also prove the  {\em tensorization property} for $\mathsf{CD}^*(K,N)$.

As an application  we conclude that
the fundamental group $\pi_1(\mathsf{M},x_0)$ of a  metric measure space $(\mathsf{M},\mathsf{d},\mathsf{m})$ is finite whenever it satisfies
locally the curvature-dimension condition $\mathsf{CD}(K,N)$ with positive $K$ and finite $N$.
\end{abstract}

\section{Introduction}

In two similar but independent approaches, the second author \cite{sa,sb} and Lott \& Villani \cite{lva} presented a concept of generalized lower Ricci bounds for metric measure spaces $(\M,\mathsf{d},\m)$.
The full strength of this concept appears if the condition $\mathsf{Ric}(\M,\mathsf{d},\m)\ge K$ is combined with a kind of upper bound $N$ for the dimension.
This leads to the so-called {\em curvature-dimension condition} $\mathsf{CD}(K,N)$ which makes sense for each pair of numbers $K\in\mathbb{R}$ and $N\in\,[1,\infty]$.

The condition $\mathsf{CD}(K,N)$ for a given metric measure space $(\M,\mathsf{d},\m)$ is formulated in terms of optimal transportation.
For general $(K,N)$ this condition is quite involved. There are two cases which lead to significant simplifications: $N=\infty$ and $K=0$.

\begin{itemize}
\item[$\rightarrow$] The condition $\mathsf{CD}(K,\infty)$, also formulated as $\mathsf{Ric}(\M,\mathsf{d},\m)\ge K$, states that for each pair
$\nu_0,\nu_1\in\mathcal{P}_\infty(\mathsf{M},\mathsf{d},\mathsf{m})$  there exists a geodesic $\nu_t=\rho_t\, m$ in $\mathcal{P}_\infty(\mathsf{M},\mathsf{d},\mathsf{m})$ connecting them such that the relative (Shannon) entropy
$$\mathsf{Ent}(\nu_t|\mathsf{m}):=\int_\M \rho_t \log \rho_t\, dm$$
is $K$-convex in $t\in[0,1]$.

Here $\mathcal{P}_\infty(\mathsf{M},\mathsf{d},\mathsf{m})$ denotes the
space of $\mathsf{m}$-absolutely continuous measures $\nu=\rho \mathsf{m}$ on $\M$ with bounded support. It is equipped with the $\mathsf{L}_2$-Wasserstein distance $\mathsf{d_W}$, see below.
\item[$\rightarrow$]
The condition $\mathsf{CD}(0,N)$ for $N\in\, (1,\infty)$ states that for each pair
$\nu_0,\nu_1\in\mathcal{P}_\infty(\mathsf{M},\mathsf{d},\mathsf{m})$  there exists a geodesic $\nu_t=\rho_t\, \mathsf{m}$ in $\mathcal{P}_\infty(\mathsf{M},\mathsf{d},\mathsf{m})$ connecting them such that the R\'enyi entropy functional
$$\mathsf{S}_{N'}(\nu_t|\mathsf{m}):=-\int_\mathsf{M}\rho_t^{1-1/N'}d\m$$
is convex in $t\in \,[0,1]$ for each $N'\ge N$.
\end{itemize}

For general  $K\in\mathbb{R}$ and $N\in\,(1,\infty)$ the condition $\mathsf{CD}(K,N)$ states that for each pair
$\nu_0,\nu_1\in\mathcal{P}_\infty(\mathsf{M},\mathsf{d},\mathsf{m})$  there exist  an optimal
coupling $\mathsf{q}$ of $\nu_0=\rho_0\mathsf{m}$ and $\nu_1=\rho_1\mathsf{m}$ and a geodesic $\nu_t=\rho_t\, m$ in $\mathcal{P}_\infty(\mathsf{M},\mathsf{d},\mathsf{m})$ connecting them such that
\begin{equation} \label{CD}
\mathsf{S}_{N'}(\nu_t|\mathsf{m})\leq-\int_{\mathsf{M}\times
\mathsf{M}}\left[\tau^{(1-t)}_{K,N'}(\mathsf{d}(x_0,x_1))\rho^{-1/N'}_0(x_0)+
\tau^{(t)}_{K,N'}(\mathsf{d}(x_0,x_1))\rho^{-1/N'}_1(x_1)\right]d\mathsf{q}(x_0,x_1)
\end{equation}
for all $t\in [0,1]$ and all $N'\geq N$.
In order to define the \textit{volume distortion coefficients} $\tau^{(t)}_{K,N}(\cdot)$ we introduce for $\theta\in\R_+$,
\begin{equation*}
\mathfrak{S}_k(\theta):=
\begin{cases}
\frac{\sin(\sqrt{k}\theta)}{\sqrt{k}\theta}& \text{if $k>0$}\\
1& \text{if $k=0$}\\
\frac{\sinh(\sqrt{-k}\theta)}{\sqrt{-k}\theta}& \text{if $k<0$}
\end{cases}
\end{equation*}
and set for $t\in[0,1]$,
\begin{equation*}
\sigma^{(t)}_{K,N}(\theta):=
\begin{cases}
\infty& \text{if $K\theta^2\geq N\pi^2$}\\
t\frac{\mathfrak{S}_{K/N}(t\theta)}{\mathfrak{S}_{K/N}(\theta)}& \text{else}
\end{cases}
\end{equation*}
as well as $\tau^{(t)}_{K,N}(\theta):=t^{1/N}\sigma^{(t)}_{K,N-1}(\theta)^{1-1/N}$.

The definitions of the condition $\mathsf{CD}(K,N)$ in \cite{sa,sb} and \cite{lva} slightly differ. We follow the notation of \cite{sa,sb}, --  except that all probability measures under consideration are now assumed to have bounded support (instead of merely having finite second moments). For non-branching spaces, all these concepts coincide.
In this case, it indeed suffices to verify (\ref{CD}) for $N'=N$ since this already implies (\ref{CD}) for all $N'\ge N$.
To simplify the presentation, we will assume for the remaining parts of the introduction that all metric measure spaces under consideration are non-branching.

\bigskip

Examples of metric measure spaces satisfying the condition $\mathsf{CD}(K,N)$ include
\begin{itemize}
\item Riemannian manifolds and weighted Riemannian spaces \cite{ov}, \cite{co}, \cite{vrs}, \cite{sc}
\item Finsler spaces \cite{oh}
\item Alexandrov spaces of generalized nonnegative sectional curvature \cite{pe}
\item Finite or infinite dimensional Gaussian spaces \cite{sa}, \cite{lv}.
\end{itemize}
Slightly modified versions are satisfied for
\begin{itemize}
\item Infinite dimensional spaces, like the Wiener space \cite{fss}, as well as for
\item Discrete spaces \cite{abs}, \cite{ol}.
\end{itemize}

\bigskip

Numerous important geometric and functional analytic estimates can be deduced from the curvature-dimension condition $\mathsf{CD}(K,N)$. Among them the Brunn-Minkowski inequality, the Bishop-Gromov volume growth estimate,  the Bonnet-Myers diameter bound, and the Lichnerowicz bound on the spectral gap.
Moreover, the condition $\mathsf{CD}(K,N)$ is {stable under convergence}.
However, two questions remained open:
\begin{itemize}
\item[$\vartriangleright$]
whether the curvature-dimension condition $\mathsf{CD}(K,N)$ for general $(K,N)$ is a {\em local property}, i.e. whether $\mathsf{CD}(K,N)$ for all subsets $\M_i$, $i\in I$, of a covering of $\M$ implies $\mathsf{CD}(K,N)$ for a given space $(\mathsf{M},\mathsf{d},\mathsf{m})$;

\item[$\vartriangleright$] whether the curvature-dimension condition $\mathsf{CD}(K,N)$ has the {\em tensorization property}, i.e. whether  $\mathsf{CD}(K,N_i)$ for each factor $\M_i$ with $i\in I$ implies $\mathsf{CD}(K,\sum_{i\in I}N_i)$ for the product space $\M=\bigotimes_{i\in I}\M_i$.

\end{itemize}
Both properties are known to be true -- or easy to verify -- in the particular cases $K=0$ and $N=\infty$. Locality of $\mathsf{CD}(K,\infty)$ was proved in \cite{sa} and, analogously, locality of $\mathsf{CD}(0,N)$ by Villani \cite{vib}. The tensorization property of $\mathsf{CD}(K,\infty)$ was proved in \cite{sa}.

\bigskip

The goal of this paper is to study metric measure spaces satisfying the local version of the curvature-dimension condition $\mathsf{CD}(K,N)$.
We prove that the {\em local} version of $\mathsf{CD}(K,N)$ is equivalent to a {\em global} condition $\mathsf{CD}^*(K,N)$, slightly weaker than the (usual, global) curvature-dimension condition.
More precisely,
$$\mathsf{CD}_{\mathsf{loc}}(K-,N)\Leftrightarrow\mathsf{CD}^*_{\mathsf{loc}}(K-,N)\Leftrightarrow\mathsf{CD}^*(K,N).$$
This so-called {\em reduced curvature-dimension condition} $\mathsf{CD}^*(K,N)$ is obtained from $\mathsf{CD}(K,N)$ by replacing the volume distortion coefficients $\tau^{(t)} _{K,N}(\cdot)$
by the slightly smaller coefficients $\sigma^{(t)} _{K,N}(\cdot)$.

Again the reduced curvature-dimension condition turns out to be {\em stable} under convergence. Moreover, we prove the {\em tensorization property} for $\mathsf{CD}^*(K,N)$.
Finally, also the reduced curvature-dimension condition allows to deduce all the geometric and functional analytic inequalities mentioned above (Bishop-Gromov, Bonnet-Myers, Lichnerowicz, etc), -- however, with slightly worse constants. Actually, this can easily be seen from the fact that for $K>0$
$$\mathsf{CD}(K,N) \Rightarrow\mathsf{CD}^*(K,N)\Rightarrow\mathsf{CD}(K^*,N)$$
with $K^*=\frac{N-1}N K$.

\medskip

As an interesting application of these results we prove that
the fundamental group $\pi_1(\mathsf{M},x_0)$ of a  metric measure space $(\mathsf{M},\mathsf{d},\mathsf{m})$ is finite whenever it satisfies
the local curvature-dimension condition $\mathsf{CD}_{\mathsf{loc}}(K,N)$ with positive $K$ and finite $N$.
Indeed, the local curvature-dimension condition for a given metric measure space $(\M,\mathsf{d},\m)$ carries over to its universal cover $(\mathsf{\hat M},\mathsf{\hat d},\mathsf{\hat m})$. The global version of the reduced curvature-dimension condition then implies a Bonnet-Myers theorem (with non-sharp constants) and thus compactness of $\mathsf{\hat M}$.

\medskip

For the purpose of comparison we point out that a similar, but slightly weaker condition than $\mathsf{CD}(K,N)$ -- the measure contraction property $\mathsf{MCP}(K,N)$ introduced in \cite{oa} and \cite{sb} -- satisfies the tensorization property due to \cite{ob} (where no assumption of non-branching metric measure spaces is used), but does not fulfill the local-to-global property according to \cite[Remark 5.6]{sb}.

\section{Reduced Curvature-Dimension Condition $\mathsf{CD}^*(K,N)$}

Throughout this paper, $(\mathsf{M},\mathsf{d},\mathsf{m})$ always denotes a \textit{metric measure space} consisting of a complete separable metric space $(\mathsf{M},\mathsf{d})$ and a locally finite measure $\mathsf{m}$ on $(\mathsf{M},\mathcal{B}(\mathsf{M}))$, that is, the volume $\mathsf{m}(B_r(x))$ of balls centered at $x$ is finite for all $x\in \mathsf{M}$ and all sufficiently small $r>0$. The metric space $(\mathsf{M},\mathsf{d})$ is called \textit{proper} if and only if every bounded closed subset of $\M$ is compact. It is called a \textit{length space} if and only if $\mathsf{d}(x,y)=\inf\mathsf{Length}(\gamma)$ for all $x,y\in\mathsf{M}$, where the infimum runs over all curves $\gamma$ in $\mathsf{M}$ connecting $x$ and $y$. Finally, it is called a \textit{geodesic space} if and only if every two points $x,y\in\mathsf{M}$ are connected by a curve $\gamma$ with $\mathsf{d}(x,y)=\mathsf{Length}(\gamma)$. Such a curve is called \textit{geodesic}. We denote by $\mathcal{G}(\mathsf{M})$ the space of geodesics $\gamma:[0,1]\rightarrow \mathsf{M}$ equipped with the topology of uniform convergence.

A \textit{non-branching} metric measure space $(\mathsf{M},\mathsf{d},\mathsf{m})$ consists of a geodesic metric space $(\mathsf{M},\mathsf{d})$ such that for every tuple $(z,x_0,x_1,x_2)$ of points in $\mathsf{M}$ for which $z$ is a midpoint of $x_0$ and $x_1$ as well as of $x_0$ and $x_2$, it follows that $x_1=x_2$.

The \textit{diameter} $\mathsf{diam}(\mathsf{M},\mathsf{d},\mathsf{m})$ of a metric measure space $(\mathsf{M},\mathsf{d},\mathsf{m})$ is defined as the diameter of its support, namely, $\mathsf{diam}(\mathsf{M},\mathsf{d},\mathsf{m}):=\sup\{\mathsf{d}(x,y): x,y\in\mathsf{supp}(\mathsf{m})\}$.

We denote by $(\mathcal{P}_2(\mathsf{M},\mathsf{d}),\mathsf{d}_{\mathsf{W}})$ the \textit{$\mathsf{L}_2$-Wasserstein space} of probability measures $\nu$ on $(\mathsf{M},\mathcal{B}(\mathsf{M}))$ with finite second moments which means that $\int_\mathsf{M}\mathsf{d}^2(x_0,x)d\nu(x)<\infty$
for some (hence all) $x_0\in \mathsf{M}$. The \textit{$\mathsf{L}_2$-Wasserstein distance} $\mathsf{d}_{\mathsf{W}}(\mu,\nu)$ between two probability measures
$\mu,\nu\in\mathcal{P}_2(\mathsf{M},\mathsf{d})$ is defined as
$$\mathsf{d}_{\mathsf{W}}(\mu,\nu)=\inf\left\{\left(\int_{\mathsf{M}\times \mathsf{M}}\mathsf{d}^2(x,y)d\mathsf{q}(x,y)\right)^{1/2}:
\text{$\mathsf{q}$ coupling of $\mu$ and $\nu$}\right\}.$$
Here the infimum ranges over all \textit{couplings} of $\mu$ and $\nu$ which are probability measures on $\mathsf{M}\times \mathsf{M}$ with marginals $\mu$ and $\nu$.

The $\mathsf{L}_2$-Wasserstein space $\mathcal{P}_2(\M,\mathsf{d})$ is a complete separable metric space. The subspace of $\mathsf{m}$-absolutely continuous measures  is denoted by $\mathcal{P}_2(\mathsf{M},\mathsf{d},\mathsf{m})$ and the subspace of $\mathsf{m}$-absolutely continuous measures with bounded support  by $\mathcal{P}_\infty(\mathsf{M},\mathsf{d},\mathsf{m})$.

The \textit{$\mathsf{L}_2$-transportation distance} $\D$ is defined for two metric measure spaces $(\M,\mathsf{d},\mathsf{m}),(\M',\mathsf{d}',\mathsf{m}')$ by
$$\D((\M,\mathsf{d},\mathsf{m}),(\M',\mathsf{d}',\mathsf{m}'))=\inf\left(\int_{\M\times \M'}\mathsf{\hat d}^2(x,y')d\mathsf{q}(x,y')\right)^{1/2}.$$
The infimum is taken over all couplings $\mathsf{q}$ of $\mathsf{m}$ and $\mathsf{m}'$ and over all couplings $\mathsf{\hat d}$ of $\mathsf{d}$ and $\mathsf{d}'$. Given two metric measure spaces $(\M,\mathsf{d},\mathsf{m})$ and $(\M',\mathsf{d}',\mathsf{m}')$, we say that a measure $\mathsf{q}$ on the product space $\M\times \M'$ is a \textit{coupling} of $\mathsf{m}$ and $\mathsf{m}'$ if and only if
$$\mathsf{q}(A\times \M')=\mathsf{m}(A) \quad \text{and} \quad \mathsf{q}(\M\times A')=\mathsf{m}'(A')$$
for all $A\in\mathcal{B}(\M)$ and all $A'\in\mathcal{B}(\M')$. We say that a pseudo-metric $\mathsf{\hat d}$ -- meaning that $\mathsf{\hat d}$ may vanish outside the diagonal -- on the disjoint union $\M\sqcup \M'$ is a \textit{coupling} of $\mathsf{d}$ and $\mathsf{d}'$ if and only if
$$\mathsf{\hat d}(x,y)=\mathsf{d}(x,y) \quad \text{and} \quad \mathsf{\hat d}(x',y')=\mathsf{d}'(x',y')$$
for all $x,y\in\mathsf{supp}(\mathsf{m})\subseteq \M$ and all $x',y'\in \mathsf{supp}(\mathsf{m}')\subseteq \M'$.

The $\mathsf{L}_2$-transportation distance $\D$ defines a complete separable length metric on the family of isomorphism classes of normalized metric measure spaces $(\M,\mathsf{d},\mathsf{m})$ satisfying $\int_\M\mathsf{d}^2(x_0,x)d\mathsf{m}(x)<\infty$ for some $x_0\in\M$.

Before we give the precise definition of the reduced curvature-dimension condition $\mathsf{CD}^*(K,N)$, we summarize two properties of the coefficients $\sigma^{(t)}_{K,N}(\cdot)$. These statements can be found in \cite{sb}.

\begin{lemma} \label{sigma}
For all $K,K'\in\R$, all $N,N'\in[1,\infty)$ and all $(t,\theta)\in[0,1]\times \R_+$,
$$\sigma^{(t)}_{K,N}(\theta)^N\cdot\sigma^{(t)}_{K',N'}(\theta)^{N'}\geq\sigma^{(t)}_{K+K',N+N'}(\theta)^{N+N'}.$$
\end{lemma}

\begin{remark}
For fixed $t\in(0,1)$ and $\theta\in(0,\infty)$ the function $(K,N)\mapsto\sigma^{(t)}_{K,N}(\theta)$ is continuous, non-decreasing in $K$ and non-increasing in $N$.
\end{remark}

\begin{definition} \label{loccurv}
Let two numbers $K,N\in\R$ with $N\geq 1$ be given.
\begin{itemize}
\item[(i)]We say that a metric measure space $(\mathsf{M},\mathsf{d},\mathsf{m})$ satisfies the reduced curvature-dimension condition $\mathsf{CD}^*(K,N)$
(globally) if and only if for all $\nu_0,\nu_1\in\mathcal{P}_\infty(\mathsf{M},\mathsf{d},\mathsf{m})$ there exist an optimal
coupling $\mathsf{q}$ of $\nu_0=\rho_0\mathsf{m}$ and $\nu_1=\rho_1\mathsf{m}$ and a geodesic
$\Gamma:[0,1]\rightarrow\mathcal{P}_\infty(\mathsf{M},\mathsf{d},\mathsf{m})$ connecting $\nu_0$ and $\nu_1$
such that
\begin{align}  \label{ren}
\mathsf{S}_{N'}(\Gamma(t)|\mathsf{m})&\leq-\int_{\mathsf{M}\times
\mathsf{M}}\left[\sigma^{(1-t)}_{K,N'}(\mathsf{d}(x_0,x_1))\rho^{-1/N'}_0(x_0)+
\sigma^{(t)}_{K,N'}(\mathsf{d}(x_0,x_1))\rho^{-1/N'}_1(x_1)\right]d\mathsf{q}(x_0,x_1)
\end{align}
for all $t\in [0,1]$ and all $N'\geq N$.
\item[(ii)]We say that $(\mathsf{M},\mathsf{d},\mathsf{m})$ satisfies the reduced curvature-dimension condition $\mathsf{CD}^*(K,N)$ locally - denoted by $\mathsf{CD}^*_{\mathsf{loc}}(K,N)$ - if and only if each point $x$ of $\mathsf{M}$ has a neighborhood $M(x)$ such that for each pair $\nu_0,\nu_1\in\mathcal{P}_\infty(\mathsf{M},\mathsf{d},\mathsf{m})$ supported in $M(x)$ there exist an optimal coupling $\mathsf{q}$ of $\nu_0$ and $\nu_1$ and a geodesic
$\Gamma:[0,1]\rightarrow\mathcal{P}_\infty(\mathsf{M},\mathsf{d},\mathsf{m})$ connecting $\nu_0$ and $\nu_1$
satisfying (\ref{ren}) for all $t\in [0,1]$ and all $N'\geq N$.
\end{itemize}
\end{definition}

\begin{remark} \label{remark}
\begin{itemize}
\item[(i)]
 For non-branching spaces, the curvature-dimension condition $\mathsf{CD}^*(K,N)$  -- which is formulated as a condition on probability measures with bounded support -- implies  property (\ref{ren}) for all measures $\nu_0,\nu_1\in\mathcal{P}_2(\mathsf{M},\mathsf{d},\mathsf{m})$. We refer to Lemma \ref{boundedsupport}. An analogous assertion holds for the condition $\mathsf{CD}(K,N)$.
\item[(ii)] In the case $K=0$, the reduced curvature-dimension condition $\mathsf{CD}^*(0,N)$ coincides with the usual one $\mathsf{CD}(0,N)$ simply because $\sigma^{(t)}_{0,N}(\theta)=t=\tau^{(t)}_{0,N}(\theta)$ for all $\theta\in\R_+$.
\item[(iii)] Note that we do not require that $\Gamma(t)$ is supported in $M(x)$ for $t\in \ ]0,1[$ in part (ii) of Definition \ref{loccurv}.
\item[(iv)] Theorem \ref{bishop} will imply that a metric measure space $(\mathsf{M},\mathsf{d},\mathsf{m})$ satisfying $\mathsf{CD}^*(K,N)$ has a proper support. In particular, the support of a metric measure space $(\mathsf{M},\mathsf{d},\mathsf{m})$ fulfilling $\mathsf{CD}^*_{\mathsf{loc}}(K,N)$ is locally compact.
\end{itemize}
\end{remark}

\begin{proposition} \label{implications}
\begin{itemize}
\item[(i)] $\mathsf{CD}(K,N)\Rightarrow$ $\mathsf{CD}^*(K,N)$:
For each metric measure space $(\mathsf{M},\mathsf{d},\mathsf{m})$, the curvature-dimension condition $\mathsf{CD}(K,N)$ for given numbers $K,N\in\R$
implies the reduced curvature-dimension condition $\mathsf{CD}^*(K,N)$.
\item[(ii)] $\mathsf{CD}^*(K,N)\Rightarrow$ $\mathsf{CD}(K^*,N)$:
Assume that $(\mathsf{M},\mathsf{d},\mathsf{m})$ satisfies the reduced curvature-dimension condition $\mathsf{CD}^*(K,N)$ for some $K>0$ and $N\geq 1$. Then $(\mathsf{M},\mathsf{d},\mathsf{m})$ satisfies $\mathsf{CD}(K^*,N)$ for $K^*=\tfrac{K(N-1)}{N}$.
\end{itemize}
\end{proposition}

\begin{proof}
\begin{itemize}
\item[(i)] Due to Lemma \ref{sigma} we have for all $K',N'\in\R$ with $N'\geq 1$ and all $(t,\theta)\in[0,1]\times \R_+$,
$$\tau^{(t)}_{K',N'}(\theta)^{N'}=t\cdot\sigma^{(t)}_{K',N'-1}(\theta)^{N'-1}
=\sigma^{(t)}_{0,1}(\theta)\cdot\sigma^{(t)}_{K',N'-1}(\theta)^{N'-1}\geq\sigma^{(t)}_{K',N'}(\theta)^{N'}$$
which means
$$\tau^{(t)}_{K',N'}(\theta)\geq\sigma^{(t)}_{K',N'}(\theta).$$
Now we consider two probability measures $\nu_0,\nu_1\in\mathcal{P}_\infty(\mathsf{M},\mathsf{d},\mathsf{m})$. Due to $\mathsf{CD}(K,N)$ there exist an optimal
coupling $\mathsf{q}$ of $\nu_0=\rho_0\mathsf{m}$ and $\nu_1=\rho_1\mathsf{m}$ and a geodesic
$\Gamma:[0,1]\rightarrow\mathcal{P}_\infty(\mathsf{M},\mathsf{d},\mathsf{m})$ connecting $\nu_0$ and $\nu_1$
such that
\begin{align*}
\mathsf{S}_{N'}(\Gamma(t)|\mathsf{m})&\leq-\int_{\mathsf{M}\times
\mathsf{M}}\left[\tau^{(1-t)}_{K,N'}(\mathsf{d}(x_0,x_1))\rho^{-1/N'}_0(x_0)+
\tau^{(t)}_{K,N'}(\mathsf{d}(x_0,x_1))\rho^{-1/N'}_1(x_1)\right]d\mathsf{q}(x_0,x_1)\\
&\leq-\int_{\mathsf{M}\times
\mathsf{M}}\left[\sigma^{(1-t)}_{K,N'}(\mathsf{d}(x_0,x_1))\rho^{-1/N'}_0(x_0)+
\sigma^{(t)}_{K,N'}(\mathsf{d}(x_0,x_1))\rho^{-1/N'}_1(x_1)\right]d\mathsf{q}(x_0,x_1)
\end{align*}
for all $t\in [0,1]$ and all $N'\geq N$.

\item[(ii)] Put $K^*:=\tfrac{K(N-1)}{N}$ and note that $K^*\leq\tfrac{K(N'-1)}{N'}$ for all $N'\geq N$. Comparing the relevant coefficients $\tau^{(t)}_{K^*,N'}(\theta)$ and $\sigma^{(t)}_{K,N'}(\theta)$, yields
\begin{equation} \label{leq}
\begin{split}
\tau^{(t)}_{K^*,N'}(\theta)=
\tau^{(t)}_{\tfrac{K(N'-1)}{N'},N'}(\theta)
=t^{1/N'}\left(\frac{\sin(t\theta\sqrt{K/N'})}{\sin(\theta\sqrt{K/N'})}\right)^{1-1/N'}
\leq\sigma^{(t)}_{K,N'}(\theta)
\end{split}
\end{equation}
for all $\theta\in\R_+$, $t\in [0,1]$ and $N'\geq N$.

According to our curvature assumption, for every $\nu_0,\nu_1\in\mathcal{P}_\infty(\mathsf{M},\mathsf{d},\mathsf{m})$  there exist an optimal coupling $\mathsf{q}$ of $\nu_0=\rho_0\mathsf{m}$ and $\nu_1=\rho_1\mathsf{m}$ and a geodesic $\Gamma:[0,1]\rightarrow\mathcal{P}_\infty(\mathsf{M},\mathsf{d},\mathsf{m})$ from $\nu_0$ to $\nu_1$ with property (\ref{ren}). From (\ref{leq}) we deduce
\begin{align*}
\mathsf{S}_{N'}(\Gamma(t)|\mathsf{m})&\leq-\int_{\mathsf{M}\times
\mathsf{M}}\left[\sigma^{(1-t)}_{K,N'}(\mathsf{d}(x_0,x_1))\rho^{-1/N'}_0(x_0)+
\sigma^{(t)}_{K,N'}(\mathsf{d}(x_0,x_1))\rho^{-1/N'}_1(x_1)\right]d\mathsf{q}(x_0,x_1)\\
&\leq-\int_{\mathsf{M}\times
\mathsf{M}}\left[\tau^{(1-t)}_{K^*,N'}(\mathsf{d}(x_0,x_1))\rho^{-1/N'}_0(x_0)+
\tau^{(t)}_{K^*,N'}(\mathsf{d}(x_0,x_1))\rho^{-1/N'}_1(x_1)\right]d\mathsf{q}(x_0,x_1)
\end{align*}
for all $t\in [0,1]$ and all $N'\geq N$. This proves property $\mathsf{CD}(K^*,N)$.
\end{itemize}
\end{proof}

A crucial property on non-branching spaces is that a mutually singular decomposition of terminal measures leads to mutually singular decompositions of $t$-midpoints. This fact was already repeatedly used in \cite{sb, lv}.  Following the advice of the referee, we include a complete proof for the readers convenience.

\begin{lemma}\label{midmutsing}
Let $(\M,\mathsf{d},\m)$ be a non-branching geodesic metric measure space. Let $\nu_0,\nu_1\in\mathcal{P}_\infty(\mathsf{M},\mathsf{d},\mathsf{m})$ and let $\nu_t$ be a $t$-midpoint of $\nu_0$ and $\nu_1$ with $t\in[0,1]$. Assume that for $n\in\mathbb{N}$ or $n=\infty$ $$\nu_i=\sum^n_{k=1}\alpha_k\nu^k_i$$
for $i=0,t,1$ and suitable $\alpha_k>0$ where $\nu^k_i$ are probability measures such that
    $\nu^k_t$ is a $t$-midpoint of $\nu^k_0$ and $\nu^k_1$ for every $k$. If the family $\left(\nu^k_0\right)_{k=1,\dots,n}$ is mutually singular, then $\left(\nu^k_t\right)_{k=1,\dots,n}$ is mutually singular as well.
\end{lemma}

\begin{proof}
We set $t_1=0$, $t_2=t$ and $t_3=1$. For $k=1,\dots,n$ we consider probability measures $\mathsf{q}^k$ on $\M^3$ with the following properties:
\begin{itemize}
\item[$\ast$]the projection on the $i$-th factor is $\nu^k_{t_i}$ for $i=1,2,3$
\item[$\ast$]for $\mathsf{q}^k$-almost every $(x_1,x_2,x_3)\in\M^3$ and every $i,j=1,2,3$
$$\mathsf{d}(x_i,x_j)=|t_i-t_j|\mathsf{d}(x_1,x_3).$$
\end{itemize}
We define $\mathsf{q}:=\sum^n_{k=1}\alpha_k\mathsf{q}^k$. Then the projection of $\mathsf{q}$ on the first and the third factor is an optimal coupling of $\nu_0$ and $\nu_1$ due to \cite[Lemma 2.11(ii)]{sb}. Assume that there exist $i,j\in\{1,\dots,n\}$ with $i\not =j$ and $(x_i,z,y_i),(x_j,z,y_j)\in\M^3$ such that
$$z\in\mathsf{supp}(\nu^i_t)\cap\mathsf{supp}(\nu^j_t)$$
and
$$\underset{\in\mathsf{supp}(\mathsf{q}^i)}{\underbrace{(x_i,z,y_i)}},\underset{\in\mathsf{supp}(\mathsf{q}^j)}{\underbrace{(x_j,z,y_j)}}\in\mathsf{supp}(\mathsf{q}).$$
Hence, $x_i\not =x_j$. Since every optimal coupling is $\mathsf{d}^2$-cyclically monotone according to \cite[Theorem 5.10]{vib}, we have
\begin{align*}
\mathsf{d}^2&(x_i,y_i)+\mathsf{d}^2(x_j,y_j)\\
&\leq\mathsf{d}^2(x_i,y_j)+\mathsf{d}^2(x_j,y_i)\\
&\leq\left[\mathsf{d}(x_i,z)+\mathsf{d}(z,y_j)\right]^2+\left[\mathsf{d}(x_j,z)+\mathsf{d}(z,y_i)\right]^2\\
&=\mathsf{d}^2(x_i,z)+\mathsf{d}^2(z,y_j)+2\mathsf{d}(x_i,z)\mathsf{d}(z,y_j)\\
&\hspace{2.5cm}+\mathsf{d}^2(x_j,z)+\mathsf{d}^2(z,y_i)+2\mathsf{d}(x_j,z)\mathsf{d}(z,y_i)\\
&=\left(t^2+(1-t)^2\right)\left[\mathsf{d}^2(x_i,y_i)+\mathsf{d}^2(x_j,y_j)\right]\\
&\hspace{2.5cm}+4t(1-t)\mathsf{d}(x_i,y_i)\mathsf{d}(x_j,y_j)\\
&\leq\left(t^2+(1-t)^2+2t(1-t)\right)\left[\mathsf{d}^2(x_i,y_i)+\mathsf{d}^2(x_j,y_j)\right]\\
&=\mathsf{d}^2(x_i,y_i)+\mathsf{d}^2(x_j,y_j).
\end{align*}
Thus, all inequalities have to be equalities. In particular,
$$\mathsf{d}(x_j,y_i)=\mathsf{d}(x_j,z)+\mathsf{d}(z,y_i),$$
meaning that $z$ is an $s$-midpoint of $x_j$ and $y_i$ for an appropriately choosen $s\in[0,1]$. Hence, there exists a tupel $(z,a_0,a_1,a_2)\in\M^4$ -- $a_1$ lying on the geodesic connecting $x_i$ and $z$, $a_2$ on the one conneting $x_j$ and $z$, $a_0$ on the one from $z$ to $y_i$ -- such that $z$ is a midpoint of $a_0$ and $a_1$ as well as of $a_0$ and $a_2$. This contradicts our assumption of non-branching metric measure spaces.
\end{proof}

We summarize two properties of the reduced curvature-dimension condition $\mathsf{CD}^*(K,N)$. The analogous results for metric measure spaces $(\mathsf{M},\mathsf{d},\mathsf{m})$ satisfying the ``original'' curvature-dimension condition $\mathsf{CD}(K,N)$ of Lott, Villani and Sturm are formulated and proved in \cite{sb}.

The first result states the uniqueness of geodesics:

\begin{proposition}[Geodesics] \label{uni}
Let $(\mathsf{M},\mathsf{d},\mathsf{m})$ be a non-branching metric measure
space satisfying the condition $\mathsf{CD}^*(K,N)$ for some numbers $K,N\in\R$.
Then for every $x\in \mathsf{supp}(\mathsf{m})\subseteq \mathsf{M}$ and $\mathsf{m}$-almost every $y\in\mathsf{M}$ - with exceptional set depending
on $x$ - there exists a unique geodesic between $x$ and $y$.

Moreover, there exists a measurable map $\gamma:\mathsf{M}\times\mathsf{M}\rightarrow\mathcal{G}(\mathsf{M})$ such that for $\mathsf{m}\otimes\mathsf{m}$-almost every $(x,y)\in\mathsf{M}\times
\mathsf{M}$ the curve $t\longmapsto\gamma_t(x,y)$ is the unique geodesic connecting $x$ and
$y$.
\end{proposition}

The second one provides equivalent characterizations of the curvature-dimension condition $\mathsf{CD}^*(K,N)$:

\begin{proposition}[Equivalent characterizations] \label{nb}
For each proper non-branching metric measure space $(\mathsf{M},\mathsf{d},\mathsf{m})$, the following statements are equivalent:
\begin{itemize}
\item[(i)]$(\mathsf{M},\mathsf{d},\mathsf{m})$ satisfies $\mathsf{CD}^*(K,N)$.
\item[(ii)]For all $\nu_0,\nu_1\in\mathcal{P}_\infty(\mathsf{M},\mathsf{d},\mathsf{m})$ there exists a geodesic $\Gamma:[0,1]\rightarrow\mathcal{P}_\infty(\mathsf{M},\mathsf{d},\mathsf{m})$
connecting $\nu_0$ and $\nu_1$ such that for all $t\in [0,1]$ and all $N'\geq N$,
\begin{equation} \label{convt}
\mathsf{S}_{N'}(\Gamma(t)|\mathsf{m})\leq\sigma^{(1-t)}_{K,N'}(\theta)\mathsf{S}_{N'}(\nu_0|\mathsf{m})+
\sigma^{(t)}_{K,N'}(\theta)\mathsf{S}_{N'}(\nu_1|\mathsf{m}),
\end{equation}
where
\begin{equation} \label{theta}
\theta:=
\begin{cases}
\inf_{x_0\in\mathcal{S}_0,x_1\in\mathcal{S}_1}\mathsf{d}(x_0,x_1),& \text{if $K\geq 0$},\\
\sup_{x_0\in\mathcal{S}_0,x_1\in\mathcal{S}_1}\mathsf{d}(x_0,x_1),& \text{if $K<0$},
\end{cases}
\end{equation}
denoting by $\mathcal{S}_0$ and $\mathcal{S}_1$ the supports of $\nu_0$ and $\nu_1$, respectively.
\item[(iii)]For all $\nu_0,\nu_1\in\mathcal{P}_\infty(\mathsf{M},\mathsf{d},\mathsf{m})$ there exists an optimal
coupling $\mathsf{q}$ of $\nu_0=\rho_0\mathsf{m}$ and $\nu_1=\rho_1\mathsf{m}$ such that
\begin{equation} \label{pointwise}
\rho^{-1/N}_t(\gamma_t(x_0,x_1))\geq\sigma^{(1-t)}_{K,N}(\mathsf{d}(x_0,x_1))\rho^{-1/N}_0(x_0)+
\sigma^{(t)}_{K,N}(\mathsf{d}(x_0,x_1))\rho^{-1/N}_1(x_1)
\end{equation}
for all $t\in [0,1]$ and $\mathsf{q}$-almost every $(x_0,x_1)\in\mathsf{M}\times\mathsf{M}$. Here for all $t\in
[0,1]$, $\rho_t$ denotes the density with respect to $\mathsf{m}$ of the push-forward measure of $\mathsf{q}$ under the
map $(x_0,x_1)\longmapsto\gamma_t(x_0,x_1)$.
\end{itemize}
\end{proposition}

\begin{proof}
(i) $\Rightarrow$ (ii): This implication follows from the fact that
$$\sigma^{(t)}_{K,N'}(\theta_\alpha)\geq\sigma^{(t)}_{K,N'}(\theta_\beta)$$
for all $t\in[0,1]$, all $N'$ and all $\theta_\alpha,\theta_\beta\in\R_+$ with $K\theta_\alpha\geq K\theta_\beta$.

(ii) $\Rightarrow$ (i): We consider two measures $\nu_0=\rho_0\mathsf{m}$, $\nu_1=\rho_1\mathsf{m}\in\mathcal{P}(B_R(o),\mathsf{d},\mathsf{m})\subseteq\mathcal{P}_\infty(\mathsf{M},\mathsf{d},\mathsf{m})$ for some $o\in \mathsf{M}$ and $R>0$ and choose an arbitrary optimal coupling $\mathsf{\tilde q}$ of them. For each $\epsilon>0$, there exists a finite covering $(C_i)_{i=1,\dots,n\in\N}$ of $M_c:=\overline{B_{2R}(o)}$ by disjoint sets $C_1,\dots,C_n$ with diameter $\leq\epsilon/2$  due to the compactness of $M_c$ which is ensured by the properness of $\mathsf{M}$. Now, we define probability measures $\nu^{ij}_0$ and $\nu^{ij}_1$ for $i,j=1,\dots,n$ on $(M_c,\mathsf{d})$ by
$$\nu^{ij}_0(A):=\frac{1}{\alpha_{ij}}\mathsf{\tilde q}((A\cap C_i)\times C_j) \quad \text{and} \quad \nu^{ij}_1(A):=\frac{1}{\alpha_{ij}} \mathsf{\tilde q}(C_i\times(A\cap C_j)),$$
provided that $\alpha_{ij}:=\mathsf{\tilde q}(C_i\times C_j)\not=0$. Then
$$\mathsf{supp}(\nu^{ij}_0)\subseteq\overline{C_i}  \quad \text{and} \quad \mathsf{supp}(\nu^{ij}_1)\subseteq\overline{C_j}.$$
By assumption there exists a geodesic $\Gamma^{ij}:[0,1]\rightarrow\mathcal{P}(M_c,\mathsf{d},\mathsf{m})$ connecting $\nu^{ij}_0=\rho^{ij}_0\mathsf{m}$ and $\nu^{ij}_1=\rho^{ij}_1\mathsf{m}$ and satisfying
\begin{align*}
\mathsf{S}_{N'}&(\Gamma^{ij}(t)|\mathsf{m})\\
&\leq-\int_{\mathsf{M}\times \mathsf{M}}\Big[\sigma^{(1-t)}_{K,N'}(\max\{\mathsf{d}(x_0,x_1)\mp\epsilon,0\})\rho^{ij}_0(x_0)^{-1/N'}+\\
&\hspace{5cm}+\sigma^{(t)}_{K,N'}(\max\{\mathsf{d}(x_0,x_1)\mp\epsilon,0\})\rho^{ij}_1(x_1)^{-1/N'}\Big]d\mathsf{q}^{ij}(x_0,x_1)
\end{align*}
for all $t\in[0,1]$ and all $N'\geq N$, with $\mp$ depending on the sign of $K$ and with $\mathsf{q}^{ij}$ being an optimal coupling of $\nu^{ij}_0$ and $\nu^{ij}_1$. We define for each $\epsilon>0$ and all $t\in[0,1]$,
$$\mathsf{q}^{(\epsilon)}:=\sum^n_{i,j=1}\alpha_{ij}\mathsf{q}^{ij} \quad \text{and} \quad \Gamma^{(\epsilon)}(t):=\sum^n_{i,j=1}\alpha_{ij}\Gamma^{ij}(t).$$
Then $\mathsf{q}^{(\epsilon)}$ is an optimal coupling of $\nu_0$ and $\nu_1$ and $\Gamma^{(\epsilon)}$ defines a geodesic connecting them. Furthermore, since $\Gamma^{ij}(t)$ is a $t$-midpoint of $\nu^{ij}_0$ and $\nu^{ij}_1$, since the $\nu^{ij}_0\otimes\nu^{ij}_1$ are mutually singular for different choices of $(i,j)\in \{1,\dots,n\}^2$ and since $(M_c,\mathsf{d},\mathsf{m})$ is non-branching, the $\Gamma^{ij}(t)$ are as well mutually singular for different choices of $(i,j)\in \{1,\dots,n\}^2$ and for each fixed $t\in [0,1]$ due to Lemma \ref{midmutsing}. Hence, for all $N'$,
$$\mathsf{S}_{N'}(\Gamma^{(\epsilon)}(t)|\mathsf{m})=\sum_{ij}\alpha^{1-1/N'}_{ij}\mathsf{S}_{N'}(\Gamma^{ij}(t)|\mathsf{m}).$$

Compactness of $(M_c,\mathsf{d})$ implies that there exists a sequence $(\epsilon(k))_{k\in\N}$ converging to 0 such that $(\mathsf{q}^{(\epsilon(k))})_{k\in\N}$ converges to some $\mathsf{q}$ and such that $(\Gamma^{(\epsilon(k))})_{k\in\N}$ converges to some geodesic $\Gamma$ in $\mathcal{P}_\infty(M_c,\mathsf{d},\mathsf{m})$. Therefore, for fixed $\varepsilon>0$, all $t\in[0,1]$ and all $N'\geq N$,
\begin{align*}
&\mathsf{S}_{N'}(\Gamma(t)|\mathsf{m})\\
&\leq\liminf_{k\to\infty}\mathsf{S}_{N'}(\Gamma^{(\epsilon(k))}(t)|\mathsf{m})\\
&\leq-\limsup_{k\to\infty}\int\Big[\sigma^{(1-t)}_{K,N'}(\max\{\mathsf{d}(x_0,x_1)\mp\varepsilon,0\})\rho^{-\tfrac{1}{N'}}_0(x_0)+\\
&\hspace{6cm}+
\sigma^{(t)}_{K,N'}(\max\{\mathsf{d}(x_0,x_1)\mp\varepsilon,0\})\rho^{-\tfrac{1}{N'}}_1(x_1)\Big]d\mathsf{q}^{(\epsilon(k))}(x_0,x_1)\\
&\leq-\int_{\mathsf{M}\times \mathsf{M}}\Big[\sigma^{(1-t)}_{K,N'}(\max\{\mathsf{d}(x_0,x_1)\mp\varepsilon,0\})\rho^{-\tfrac{1}{N'}}_0(x_0)+\\
&\hspace{6cm}+
\sigma^{(t)}_{K,N'}(\max\{\mathsf{d}(x_0,x_1)\mp\varepsilon,0\})\rho^{-\tfrac{1}{N'}}_1(x_1)\Big]d\mathsf{q}(x_0,x_1)
\end{align*}
where the proof of the last inequality is similar to the proof of [Stu06b, Lemma 3.3]. In the limit $\varepsilon\to 0$ the claim follows due to the theorem of monotone convergence.

The equivalence (i) $\Leftrightarrow$ (iii) is obtained by following the arguments of the proof of [Stu06b, Proposition 4.2] replacing the coefficients $\tau^{(t)}_{K,N}(\cdot)$ by $\sigma^{(t)}_{K,N}(\cdot)$.
\end{proof}

\begin{remark}
To be honest, we suppressed an argument in the proof of Proposition \ref{nb}, (ii) $\Rightarrow$ (i): In fact, the compactness of $(M_c,\mathsf{d})$ implies the compactness of $\mathcal{P}(M_c,\mathsf{d})$ and therefore, we can deduce the existence of a limit $\Gamma$ of $(\Gamma^{(\epsilon(k))})_{k\in\N}$ - using the same notation as in the above proof - in $\mathcal{P}(M_c,\mathsf{d})$! A further observation ensures that $\Gamma$ is not only in $\mathcal{P}(M_c,\mathsf{d})$ but also in $\mathcal{P}_\infty(M_c,\mathsf{d},\mathsf{m})$ - as claimed in the above proof: The characterizing inequality of $\mathsf{CD}^*(K,N)$ implies the characterizing inequality of the property $\mathsf{Curv}(\mathsf{M},\mathsf{d},\mathsf{m})\geq K$ (at this point we refer to \cite{sa},\cite{sb}). Thus, the geodesic $\Gamma$ satisfies
$$\mathsf{Ent}(\Gamma(t)|\mathsf{m})\leq(1-t)\mathsf{Ent}(\Gamma(0)|\mathsf{m})+t\mathsf{Ent}(\Gamma(1)|\mathsf{m})-\frac{K}{2}t(1-t)\mathsf{d}^2_{\mathsf{W}}(\Gamma(0),\Gamma(1))$$
for all $t\in[0,1]$. This implies that $\mathsf{Ent}(\Gamma(t)|\mathsf{m})<+\infty$ and consequently, $\Gamma(t)\in\mathcal{P}_\infty(M_c,\mathsf{d},\mathsf{m})$ for all $t\in[0,1]$.
In the sequel, we will use similar arguments from time to time without emphasizing on them explicitly.
\end{remark}

\begin{proposition} [Midpoints] \label{midpoints}
A proper non-branching metric measure space $(\mathsf{M},\mathsf{d},\mathsf{m})$ satisfies $\mathsf{CD}^*(K,N)$ if and only if for all $\nu_0,\nu_1\in\mathcal{P}_\infty(\mathsf{M},\mathsf{d},\mathsf{m})$ there exists a midpoint $\eta\in\mathcal{P}_\infty(\mathsf{M},\mathsf{d},\mathsf{m})$ of $\nu_0$ and $\nu_1$ satisfying
\begin{equation} \label{convmid}
\mathsf{S}_{N'}(\eta|\mathsf{m})\leq\sigma^{(1/2)}_{K,N'}(\theta)\mathsf{S}_{N'}(\nu_0|\mathsf{m})+
\sigma^{(1/2)}_{K,N'}(\theta)\mathsf{S}_{N'}(\nu_1|\mathsf{m}),
\end{equation}
for all $N'\geq N$ where $\theta$ is defined as in (\ref{theta}).
\end{proposition}

\begin{proof}
We only consider the case $K>0$. The general case requires analogous calculations. Due to Proposition \ref{nb}, we have to prove that the existence of midpoints with property (\ref{convmid}) for all $N'\geq N$
implies the existence of geodesics satisfying property (\ref{convt}) for all $N'\geq N$.
Given $\Gamma(0):=\nu_0$ and $\Gamma(1):=\nu_1$, we define $\Gamma(\tfrac{1}{2})$ as a midpoint of $\Gamma(0)$ and $\Gamma(1)$
with property (\ref{convmid}) for all $N'\geq N$. Then we define $\Gamma(\tfrac{1}{4})$ as a midpoint of $\Gamma(0)$ and $\Gamma(\tfrac{1}{2})$
satisfying (\ref{convmid}) for all $N'\geq N$ and accordingly, $\Gamma(\tfrac{3}{4})$ as a midpoint of $\Gamma(\tfrac{1}{2})$ and $\Gamma(1)$
with (\ref{convmid}) for all $N'\geq N$. By iterating this procedure, we obtain $\Gamma(t)$ for all dyadic $t=l2^{-k}\in[0,1]$ for $k\in\N$ and odd $l=0,\dots,2^k$
with
\begin{align*}
\mathsf{S}_{N'}&\left(\Gamma\left(l2^{-k}\right)|\mathsf{m}\right)\leq\\
&\leq\sigma^{(1/2)}_{K,N'}\left(2^{-k+1}\theta\right)\mathsf{S}_{N'}\left(\Gamma\left((l-1)2^{-k}\right)|\mathsf{m}\right)+
\sigma^{(1/2)}_{K,N'}\left(2^{-k+1}\theta\right)\mathsf{S}_{N'}\left(\Gamma\left((l+1)2^{-k}\right)|\mathsf{m}\right),
\end{align*}
for all $N'\geq N$ where $\theta$ is defined as above.

Now, we consider $k>0$. By induction, we are able to pass from level $k-1$ to level $k$: Assuming that $\Gamma(t)$ satisfies property (\ref{convt}) for all $t=l2^{-k+1}\in[0,1]$ and all $N'\geq N$, we have for an odd number $l\in\{0,\dots,2^{-k}\}$,
\begin{align*}
\mathsf{S}_{N'}&\left(\Gamma\left(l2^{-k}\right)|\mathsf{m}\right)\leq\\
&\leq\sigma^{(1/2)}_{K,N'}\left(2^{-k+1}\theta\right)\mathsf{S}_{N'}\left(\Gamma\left((l-1)2^{-k}\right)|\mathsf{m}\right)+
\sigma^{(1/2)}_{K,N'}\left(2^{-k+1}\theta\right)\mathsf{S}_{N'}\left(\Gamma\left((l+1)2^{-k}\right)|\mathsf{m}\right)\\
&\leq\sigma^{(1/2)}_{K,N'}\left(2^{-k+1}\theta\right)\left[\sigma^{\left(1-(l-1)2^{-k}\right)}_{K,N'}(\theta)\mathsf{S}_{N'}(\Gamma(0)|\mathsf{m})+
\sigma^{\left((l-1)2^{-k}\right)}_{K,N'}(\theta)\mathsf{S}_{N'}(\Gamma(1)|\mathsf{m})\right]+\\
&+\sigma^{(1/2)}_{K,N'}\left(2^{-k+1}\theta\right)\left[\sigma^{\left(1-(l+1)2^{-k}\right)}_{K,N'}(\theta)\mathsf{S}_{N'}(\Gamma(0)|\mathsf{m})+
\sigma^{\left((l+1)2^{-k}\right)}_{K,N'}(\theta)\mathsf{S}_{N'}(\Gamma(1)|\mathsf{m})\right]
\end{align*}
for all $N'\geq N$. Calculating the prefactor of $\mathsf{S}_{N'}(\Gamma(0)|\mathsf{m})$ yields
\begin{align*}
&\sigma^{(1/2)}_{K,N'}\left(2^{-k+1}\theta\right)\sigma^{\left(1-(l-1)2^{-k}\right)}_{K,N'}(\theta)+
\sigma^{(1/2)}_{K,N'}\left(2^{-k+1}\theta\right)\sigma^{\left(1-(l+1)2^{-k}\right)}_{K,N'}(\theta)=\\
&=\frac{\sin\left(2^{-k}\theta\sqrt{K/N'}\right)\cdot\left[\sin\left(\left(1-(l-1)2^{-k}\right)\theta\sqrt{K/N'}\right)+
\sin\left(\left(1-(l+1)2^{-k}\right)\theta\sqrt{K/N'}\right)\right]}{\sin\left(2^{-k+1}\theta\sqrt{K/N'}\right)\sin\left(\theta\sqrt{K/N'}\right)}\\
&=\frac{2\sin\left(\left(1-l2^{-k}\right)\theta\sqrt{K/N'}\right)\cos\left(2^{-k}\theta\sqrt{K/N'}\right)}
{2\cos\left(2^{-k}\theta\sqrt{K/N'}\right)\sin\left(\theta\sqrt{K/N'}\right)}=\\
&=\frac{\sin\left(\left(1-l2^{-k}\right)\theta\sqrt{K/N'}\right)}{\sin\left(\theta\sqrt{K/N'}\right)}=\sigma^{\left(1-l2^{-k}\right)}_{K,N'}(\theta),
\end{align*}
and calculating the one of $\mathsf{S}_{N'}(\Gamma(1)|\mathsf{m})$ gives
\begin{align*}
\sigma^{(1/2)}_{K,N'}&\left(2^{-k+1}\theta\right)\sigma^{\left((l-1)2^{-k}\right)}_{K,N'}(\theta)+
\sigma^{(1/2)}_{K,N'}\left(2^{-k+1}\theta\right)\sigma^{\left((l+1)2^{-k}\right)}_{K,N'}(\theta)=\\
&=\frac{\sin\left(2^{-k}\theta\sqrt{K/N'}\right)\cdot\left[\sin\left((l-1)2^{-k}\theta\sqrt{K/N'}\right)+
\sin\left((l+1)2^{-k}\theta\sqrt{K/N'}\right)\right]}{\sin\left(2^{-k+1}\theta\sqrt{K/N'}\right)\sin\left(\theta\sqrt{K/N'}\right)}\\
&=\frac{2\sin\left(l2^{-k}\theta\sqrt{K/N'}\right)\cos\left(2^{-k}\theta\sqrt{K/N'}\right)}
{2\cos\left(2^{-k}\theta\sqrt{K/N'}\right)\sin\left(\theta\sqrt{K/N'}\right)}=\\
&=\frac{\sin\left(l2^{-k}\theta\sqrt{K/N'}\right)}{\sin\left(\theta\sqrt{K/N'}\right)}=\sigma^{\left(l2^{-k}\right)}_{K,N'}(\theta).
\end{align*}
Combining the above results leads to property (\ref{convt}),
$$\mathsf{S}_{N'}\left(\Gamma\left(l2^{-k}\right)|\mathsf{m}\right)\leq\sigma^{\left(1-l2^{-k}\right)}_{K,N'}(\theta)\mathsf{S}_{N'}(\Gamma(0)|\mathsf{m})+
\sigma^{\left(l2^{-k}\right)}_{K,N'}(\theta)\mathsf{S}_{N'}(\Gamma(1)|\mathsf{m})$$
for all $N'\geq N$. The continuous extension of $\Gamma(t)$ -- $t$ dyadic -- yields the desired geodesic due to the lower semi-continuity of the R\'enyi entropy.
\end{proof}

\begin{lemma} \label{boundedsupport}
Fix two real parameters $K$ and $N\geq 1$. If $(\mathsf{M},\mathsf{d},\mathsf{m})$ is non-branching then the reduced curvature-dimension condition $\mathsf{CD}^*(K,N)$ implies that for all $\nu_0,\nu_1\in\mathcal{P}_2(\mathsf{M},\mathsf{d},\mathsf{m})$ there exist an optimal
coupling $\mathsf{q}$ of $\nu_0=\rho_0\mathsf{m}$ and $\nu_1=\rho_1\mathsf{m}$ and a geodesic
$\Gamma:[0,1]\rightarrow\mathcal{P}_2(\mathsf{M},\mathsf{d},\mathsf{m})$ connecting $\nu_0$ and $\nu_1$ and satisfying (\ref{ren}) for all $N'\geq N$.
\end{lemma}

\begin{proof}
We assume that $(\mathsf{M},\mathsf{d},\mathsf{m})$ satisfies $\mathsf{CD}^*(K,N)$. Fix a covering of $M$ by mutual disjoint, bounded sets $L_i, i\in{\mathbb N}$. Let $\nu_0=\rho_0\mathsf{m},\nu_1=\rho_1\mathsf{m}\in\mathcal{P}_2(\mathsf{M},\mathsf{d},\mathsf{m})$ and an optimal coupling ${\tilde{q}}$ of $\nu_0$ and $\nu_1$ be given.
Define probability
measures $\nu_0^{ij}, \nu_1^{ij}\in\mathcal{P}_\infty(\mathsf{M},\mathsf{d},\mathsf{m})$ for $i,j\in{\mathbb N}$ by
\begin{eqnarray*}
\label{7a}
\nu_0^{ij}(A):=\frac1{\alpha_{ij}} \tilde q((A\cap L_i)\times L_j)\quad\mbox{ and }\quad
\nu_1^{ij}(A):=\frac1{\alpha_{ij}} \tilde q(L_i\times(A\cap L_j))
\end{eqnarray*}
provided $\alpha_{ij}:=\tilde q(L_i\times L_j)\not=0$.
According to $\mathsf{CD}^*(K,N)$, for each pair $i,j\in\N$, there exist an optimal coupling ${q}_{ij}$ of $\nu^{ij}_0=\rho^{ij}_0\mathsf{m}$ and $\nu^{ij}_1=\rho^{ij}_1\mathsf{m}$ and a geodesic $\Gamma^{ij}:[0,1]\to\mathcal{P}_\infty(\mathsf{M},\mathsf{d},\mathsf{m})$ joining them such that
\begin{align*}
&\mathsf{S}_{N'}(\Gamma^{ij}(t)|\mathsf{m})\leq\\
&\leq-\int_{\mathsf{M}\times
\mathsf{M}}\left[\sigma^{(1-t)}_{K,N'}(\mathsf{d}(x_0,x_1))\rho^{ij}_0(x_0)^{-1/N'}+
\sigma^{(t)}_{K,N'}(\mathsf{d}(x_0,x_1))\rho^{ij}_1(x_1)^{-1/N'}\right]d {q}_{ij}(x_0,x_1)
\end{align*}
for all $t\in[0,1]$ and all $N'\geq N$.
Define
\begin{eqnarray*}\label{7b}
q:=\sum_{i,j=1}^\infty \alpha_{ij} q^{ij}, \quad
\Gamma_t:=\sum_{i,j=1}^\infty \alpha_{ij} \Gamma^{ij}_t.
\end{eqnarray*}
Then $q$ is an optimal coupling of $\nu_0$ and $\nu_1$ and $\Gamma$
is a geodesic connecting them. Moreover, since the
$\nu_0^{ij}\otimes \nu_1^{ij}$ for different choices of
$(i,j)\in\mathbb{N}^2$ are mutually singular and since $M$ is
non-branching, also  the $\Gamma_t^{ij}$ for different choices of
$(i,j)\in\{1,\ldots,n\}^2$ are mutually singular, Lemma \ref{midmutsing}
(for each fixed $t\in[0,1]$). Hence,
$$S_{N'}(\Gamma_t|m)=\sum_{i,j=1}^\infty \alpha_{ij}^{1-1/N'}\cdot S_{N'}(\Gamma_t^{ij}|m)$$
and one simply may sum up both sides of the previous inequality
-- multiplied by $\alpha_{ij}^{1-1/N'}$ -- to obtain the claim.
\end{proof}

\begin{remark}
Let us point out that the same arguments prove that on non-branching spaces the curvature-dimension condition $\mathsf{CD}(K,N)$ as formulated in this paper -- which requires only conditions on probability measures with bounded support -- implies the analogous condition in the second author's previous paper \cite{sb} (where conditions on all  probability measures with finite second moments had been imposed).
\end{remark}

\begin{remark}
The curvature-dimension condition $\mathsf{CD}(K,N)$ does not imply the non-branching property. For instance, Banach spaces satisfy $\mathsf{CD}(0,N)$ whereas they are not always non-branching. Moreover, even in the special case of limits of Riemannian manifolds with uniform lower Ricci curvature bounds, it is not known whether they are non-branching or not.
\end{remark}

\section{Stability under Convergence}

\begin{theorem} \label{stab}
Let $((\mathsf{M}_n,\mathsf{d}_n,\mathsf{m}_n))_{n\in\N}$ be a sequence of normalized metric measure spaces with the property that for each $n\in\N$ the space $(\mathsf{M}_n,\mathsf{d}_n,\mathsf{m}_n)$ satisfies the reduced curvature-dimension condition $\mathsf{CD}^*(K_n,N_n)$. Assume that for $n\to\infty$,
$$(\mathsf{M}_n,\mathsf{d}_n,\mathsf{m}_n)\overset{\D}{\to}(\M,\mathsf{d},\mathsf{m})$$
as well as $(K_n,N_n)\to(K,N)$ for some $(K,N)\in\R^2$. Then the space $(\M,\mathsf{d},\mathsf{m})$ fulfills $\mathsf{CD}^*(K,N)$.
\end{theorem}

\begin{proof}
The proof essentially follows the line of argumentation in \cite[Theorem 3.1]{sb} with two modifications:
\begin{itemize}
\item[$\ast$] The coefficients $\tau^{(t)}_{K,N}(\cdot)$ will be replaced by $\sigma^{(t)}_{K,N}(\cdot)$.
\item[$\ast$] The assumption of a uniform upper bound $L_0<L_{\mathrm{max}}$ on the diameters will be removed. (Here $L_{\mathrm{max}}$ will be $\pi\sqrt{\tfrac{N}{K}}$ for $K>0$, previously it was $\pi\sqrt{\tfrac{N-1}{K}}$.)
\end{itemize}
\begin{itemize}
\item[(i)] Let us firstly observe that $\mathsf{CD}^*(K_n,N_n)$ with $K_n\to K$ and $N_n\to N$ implies that the spaces $(\mathsf{M}_n,\mathsf{d}_n,\mathsf{m}_n)$ have the \textquoteleft doubling property\textquoteright \ with a common doubling constant $C$ on subsets $M'_n\subseteq\mathsf{supp}(\m_n)$ with uniformly bounded diameter $\theta$ (see \cite[Corollary 2.4]{sb} and also Theorem \ref{bishop}). This version of the doubling property is stable under $\D$-convergence due to \cite[Theorem 3.15]{sa} and thus also holds on bounded sets $M'\subseteq\mathsf{supp}(\m)$. Therefore, $\mathsf{supp}(\m)$ is proper.
\item[(ii)] Choose $\bar{N}>N$ and $\bar{K}<K$ and put $\bar{L}:=\pi\sqrt{\tfrac{\bar{N}}{\bar{K}}}$ as well as $L:=\pi\sqrt{\tfrac{N}{K}}$ provided that $\bar{K}>0$ and $K>0$. Otherwise, $\bar{L}=\infty$, $L=\infty$. Then
$$\max\left\{\tfrac{\partial}{\partial\theta}\sigma^{(s)}_{K',N'}(\theta):s\in[0,1],K'\leq\bar{K},N'\geq\bar{N},\theta\in\left[0,\tfrac{L+\bar{L}}{2}\right]\right\}$$  is bounded.
\item[(iii)] For each $n\in\N$, $\mathsf{diam}(\mathsf{supp}(\m_n))\leq L_n:=\pi\sqrt{\tfrac{N_n}{K_n}}$
due to Corollary \ref{bonnet}. In particular, given $\bar{K}$, $\bar{N}$ as above
$$\mathsf{diam}(\mathsf{supp}(\m_n))\leq\tfrac{L+\bar{L}}{2}$$
for all sufficiently large $n\in\N$. The latter implies
$$\mathsf{diam}(\mathsf{supp}(\m))\leq\tfrac{L+\bar{L}}{2}$$
according to \cite[Theorem 3.16]{sa}.
\item[(iv)] Let us now follow the proof in \cite[Theorem 3.1]{sb}. In short, we consider $\nu_0,\nu_1\in\mathcal{P}_\infty(\mathsf{M},\mathsf{d},\mathsf{m})$ and approximate them by probability measures $\nu_{0,n}$ and $\nu_{1,n}$ in $\mathcal{P}_\infty(\mathsf{M}_n,\mathsf{d}_n,\mathsf{m}_n)$ satisfying the relevant equation (\ref{ren}) with an optimal coupling $\mathsf{q}_n$ and a geodesic $\Gamma_{t,n}$ due to the reduced curvature-dimension condition on $(\mathsf{M}_n,\mathsf{d}_n,\mathsf{m}_n)$. Via a map $\mathcal{Q}:\mathcal{P}_2(\mathsf{M}_n,\mathsf{d}_n,\mathsf{m}_n)\rightarrow\mathcal{P}_2(\mathsf{M},\mathsf{d},\mathsf{m})$ introduced in \cite[Lemma 4.19]{sa} we define an \textquoteleft $\varepsilon$-approximative\textquoteright \ geodesic $\Gamma^\varepsilon_{t}:=\mathcal{Q}(\Gamma_{t,n})$ from $\nu_0$ to $\nu_1$ satisfying (\ref{ren}) for an \textquoteleft $\varepsilon$-approximative\textquoteright \ coupling $\mathsf{q}^\varepsilon$ of $\nu_0$ and $\nu_1$.
\item[(v)] The properness of $\mathsf{supp}(\m)$ implies that $\Gamma^\varepsilon_{t}$ and $\mathsf{q}^\varepsilon$ are tight (i.e. essentially supported on compact sets -- uniformly in $\varepsilon$) which yields the existence of accumulation points $\bar{\Gamma}_t$ and $\bar{\mathsf{q}}$ satisfying (\ref{ren}) -- with $K'$ in the place of $K$ -- for all $K'\leq\bar{K}$ and all $N'\geq\bar{N}$.
\item[(vi)] Choosing sequences $\bar{N}_l\searrow N$ and $\bar{K}_l\nearrow K$ and again passing to the limits $\Gamma_t=\lim_l\bar{\Gamma}^l_t$ and $\mathsf{q}=\lim_l\bar{\mathsf{q}} ^l$ we obtain an optimal coupling $\mathsf{q}$ and a geodesic $\Gamma$ satisfying (\ref{ren}) for all $K'<K$ and all $N'>N$. Finally, continuity of all the involved terms in $K'$ and $N'$ proves the claim.
\end{itemize}
\end{proof}

\begin{remark}
The previous proof demonstrates that in the analogous formulation of the stability result for $\mathsf{CD}(K,N)$ in \cite[Theorem 3.1]{sb} the assumption
$$\limsup_{n\to\infty}\frac{K_nL^2_n}{N_n-1}<\pi$$
is unnecessary.
\end{remark}
\section{Tensorization}

\begin{theorem}[Tensorization] \label{tensorization}
Let $(\M_i,\mathsf{d}_i,\mathsf{m}_i)$ be non-branching metric measure spaces satisfying the reduced curvature-dimension condition $\mathsf{CD}^*(K,N_i)$ with two real parameters $K$ and $N_i\geq 1$ for $i=1,\dots,k$ with $k\in\N$. Then
$$(\mathsf{M},\mathsf{d},\mathsf{m}):=\bigotimes^k_{i=1}(\M_i,\mathsf{d}_i,\mathsf{m}_i)$$
fulfills $\mathsf{CD}^*\left(K,\sum^k_{i=1}N_i\right)$.
\end{theorem}

\begin{proof}
Without restriction we assume that $k=2$. We consider $\nu_0=\rho_0\mathsf{m},\nu_1=\rho_1\mathsf{m}\in\mathcal{P}_\infty(\mathsf{M},\mathsf{d},\mathsf{m})$. In the first step, we treat the special case
$$\nu_0=\nu^{(1)}_0\otimes\nu^{(2)}_0 \quad \mathrm{and} \quad \nu_1=\nu^{(1)}_1\otimes\nu^{(2)}_1$$
with $\nu^{(i)}_0=\rho^{(i)}_0\mathsf{m}_i,\nu^{(i)}_1=\rho^{(i)}_1\mathsf{m}_i\in\mathcal{P}_\infty(\M_i,\mathsf{d}_i,\mathsf{m}_i)$ for $i=1,2$. According to our curvature assumption, there exists an optimal coupling $\mathsf{q}_i$ of $\nu^{(i)}_0$ and $\nu^{(i)}_1$ such that
\begin{align*}
\rho^{(i)}_t&\left(\gamma^{(i)}_t\left(x^{(i)}_0,x^{(i)}_1\right)\right)^{-1/N_i}\geq\\
&\geq\sigma^{(1-t)}_{K,N_i}\left(\mathsf{d}_i\left(x^{(i)}_0,x^{(i)}_1\right)\right)\rho^{(i)}_0\left(x^{(i)}_0\right)^{-1/N_i}+
\sigma^{(t)}_{K,N_i}\left(\mathsf{d}_i\left(x^{(i)}_0,x^{(i)}_1\right)\right)\rho^{(i)}_1\left(x^{(i)}_1\right)^{-1/N_i}
\end{align*}
for all $t\in [0,1]$ and $\mathsf{q}_i$-almost every $\left(x^{(i)}_0,x^{(i)}_1\right)\in\M_i\times\M_i$ with $i=1,2$. As in Proposition \ref{nb}, for all $t\in[0,1]$, $\rho^{(i)}_t$ denotes the density with respect to $\mathsf{m}_i$ of the push-forward measure of $\mathsf{q}_i$ under the
map $\left(x^{(i)}_0,x^{(i)}_1\right)\mapsto\gamma^{(i)}_t\left(x^{(i)}_0,x^{(i)}_1\right)$ for $i=1,2$.
We introduce the map
\begin{align*}
\mathsf{T}:\mathsf{M}_1\times\mathsf{M}_1\times\mathsf{M}_2\times\mathsf{M}_2&\rightarrow\mathsf{M}_1\times\mathsf{M}_2\times\mathsf{M}_1\times\mathsf{M}_2
=\mathsf{M}\times\mathsf{M}\\
\left(x^{(1)}_0,x^{(1)}_1,x^{(2)}_0,x^{(2)}_1\right)&\mapsto\left(x^{(1)}_0,x^{(2)}_0,x^{(1)}_1,x^{(2)}_1\right),
\end{align*}
we put $\mathsf{\tilde q}:=\mathsf{q}_1\otimes\mathsf{q}_2$ and define $\mathsf{q}$ as the push-forward measure of $\mathsf{\tilde q}$ under the map $\mathsf{T}$, that means $\mathsf{q}:=\mathsf{T}_*\mathsf{\tilde q}$. Then $\mathsf{q}$ is an optimal coupling of $\nu_0$ and $\nu_1$ and for all $t\in[0,1]$, $\rho_t(x,y):=\rho^{(1)}_t(x)\cdot\rho^{(2)}_t(y)$ is the density with respect to $\mathsf{m}$ of the push-forward measure of $\mathsf{q}$ under the map
\begin{align*}
\gamma_t:\mathsf{M}\times\mathsf{M}&\rightarrow\mathsf{M}=\mathsf{M}_1\times\mathsf{M}_2\\
\left(x^{(1)}_0,x^{(2)}_0,x^{(1)}_1,x^{(2)}_1\right)&\mapsto\left(\gamma^{(1)}_t\left(x^{(1)}_0,x^{(1)}_1\right),\gamma^{(2)}_t\left(x^{(2)}_0,x^{(2)}_1\right)\right).
\end{align*}
Moreover, for $\mathsf{q}$-almost every $x_0=\left(x^{(1)}_0,x^{(2)}_0\right),x_1=\left(x^{(1)}_1,x^{(2)}_1\right)\in\mathsf{M}$ and all $t\in[0,1]$, it holds that
\begin{align*}
&\sigma^{(1-t)}_{K,N_1+N_2}(\mathsf{d}(x_0,x_1))\rho_0(x_0)^{-1/(N_1+N_2)}+\sigma^{(t)}_{K,N_1+N_2}(\mathsf{d}(x_0,x_1))\rho_1(x_1)^{-1/(N_1+N_2)}=\\
&=\sigma^{(1-t)}_{K,N_1+N_2}(\mathsf{d}(x_0,x_1))\rho^{(1)}_0\left(x^{(1)}_0\right)^{-1/(N_1+N_2)}\cdot\rho^{(2)}_0\left(x^{(2)}_0\right)^{-1/(N_1+N_2)}\quad+\\
&\hspace{3cm}+\sigma^{(t)}_{K,N_1+N_2}(\mathsf{d}(x_0,x_1))\rho^{(1)}_1\left(x^{(1)}_1\right)^{-1/(N_1+N_2)}\cdot\rho^{(2)}_1\left(x^{(2)}_1\right)^{-1/(N_1+N_2)}\\
&\leq\prod^2_{i=1}\sigma^{(1-t)}_{K,N_i}\left(\mathsf{d}_i\left(x^{(i)}_0,x^{(i)}_1\right)\right)^{N_i/(N_1+N_2)}\rho^{(i)}_0\left(x^{(i)}_0\right)^{-1/(N_1+N_2)}\quad+\\
&\hspace{3cm}+\prod^2_{i=1}\sigma^{(t)}_{K,N_i}\left(\mathsf{d}_i\left(x^{(i)}_0,x^{(i)}_1\right)\right)^{N_i/(N_1+N_2)}\rho^{(i)}_1\left(x^{(i)}_1\right)^{-1/(N_1+N_2)}\\
&\leq\prod^2_{i=1}\biggl[\sigma^{(1-t)}_{K,N_i}\left(\mathsf{d}_i\left(x^{(i)}_0,x^{(i)}_1\right)\right)\rho^{(i)}_0\left(x^{(i)}_0\right)^{-1/N_i}+\\
&\hspace{3cm}+\sigma^{(t)}_{K,N_i}\left(\mathsf{d}_i\left(x^{(i)}_0,x^{(i)}_1\right)\right)\rho^{(i)}_1\left(x^{(i)}_1\right)^{-1/N_i}
\biggr]^{N_i/(N_1+N_2)}\\
&\leq\prod^2_{i=1}\rho^{(i)}_t\left(\gamma^{(i)}_t\left(x^{(i)}_0,x^{(i)}_1\right)\right)^{-1/(N_1+N_2)}\\
&=\rho_t\left(\gamma^{(1)}_t\left(x^{(1)}_0,x^{(1)}_1\right),\gamma^{(2)}_t\left(x^{(2)}_0,x^{(2)}_1\right)\right)^{-1/(N_1+N_2)}=\rho_t\left(\gamma_t(x_0,x_1)\right)^{-1/(N_1+N_2)}.
\end{align*}
In this chain of inequalities, the second one follows from Lemma \ref{sigma} and the third one from H\"older's inequality.

\bigskip

In the second step, we consider $o\in\mathsf{supp(\mathsf{m})}$ and $R>0$ and set $M_b:=B_R(o)\cap\mathsf{supp(\mathsf{m})}$ as well as $M_{c}:=\overline{B_{2R}(o)}\cap\mathsf{supp(\mathsf{m})}$. We consider arbitrary probability measures $\nu_0,\nu_1\in\mathcal{P}_\infty(M_b,\mathsf{d},\mathsf{m})$ and $\varepsilon>0$. There exist
$$\nu^\varepsilon_0=\rho^\varepsilon_0m=\frac{1}{n}\sum^n_{j=1}\nu^\varepsilon_{0,j}$$
with mutually singular product measures $\nu^\varepsilon_{0,j}$ and
$$\nu^\varepsilon_1=\rho^\varepsilon_1m=\frac{1}{n}\sum^n_{j=1}\nu^\varepsilon_{1,j}$$
with mutually singular product measures $\nu^\varepsilon_{1,j}$ for $j=1,\dots,n$ and $n\in\N$ such that
\begin{align*}
\mathsf{S}_{N_1+N_2}\left(\nu^\varepsilon_0|\mathsf{m}\right)&\leq\mathsf{S}_{N_1+N_2}(\nu_0|\mathsf{m})+\varepsilon,\\
\mathsf{S}_{N_1+N_2}\left(\nu^\varepsilon_1|\mathsf{m}\right)&\leq\mathsf{S}_{N_1+N_2}(\nu_1|\mathsf{m})+\varepsilon
\end{align*}
as well as
$$\mathsf{d_W}\left(\nu_0,\nu^\varepsilon_0\right)\leq\varepsilon,\quad\mathsf{d_W}\left(\nu_1,\nu^\varepsilon_1\right)\leq\varepsilon$$
and
$$\mathsf{d_W}\left(\nu^\varepsilon_0,\nu^\varepsilon_1\right)\geq\left[\frac{1}{n}\sum^n_{j=1}\mathsf{d}^2_\mathsf{W}\left(\nu^\varepsilon_{0,j},\nu^\varepsilon_{1,j}\right)\right]^{1/2}-\varepsilon.$$
Moreover,
\begin{align*}
\theta&:=
\begin{cases}
\underset{x_1\in\mathrm{supp}(\nu_1)}{\underset{x_0\in\mathrm{supp}(\nu_0),}{\inf}}\mathsf{d}(x_0,x_1)\leq\underset{x_1\in\mathrm{supp}\left(\nu^\varepsilon_{1,j}\right)}{\underset{x_0\in\mathrm{supp}\left(\nu^\varepsilon_{0,j}\right),}{\inf}}\mathsf{d}(x_0,x_1),& \text{if $K\geq 0$},\\
\underset{x_1\in\mathrm{supp}(\nu_1)}{\underset{x_0\in\mathrm{supp}(\nu_0),}{\sup}}\mathsf{d}(x_0,x_1)\geq\underset{x_1\in\mathrm{supp}\left(\nu^\varepsilon_{1,j}\right)}{\underset{x_0\in\mathrm{supp}\left(\nu^\varepsilon_{0,j}\right),}{\sup}}\mathsf{d}(x_0,x_1),& \text{if $K<0$}.
\end{cases}
\end{align*}
Since $\nu^\varepsilon_0$ is the sum of mutually singular measures $\nu^\varepsilon_{0,j}$ for $j=1,\dots,n$,
$$\mathsf{S}_{N_1+N_2}\left(\nu^\varepsilon_0|\mathsf{m}\right)=\left(\frac{1}{n}\right)^{1-1/(N_1+N_2)}\sum^n_{j=1}\mathsf{S}_{N_1+N_2}\left(\nu^\varepsilon_{0,j}|\mathsf{m}\right)$$
and analogously,
$$\mathsf{S}_{N_1+N_2}\left(\nu^\varepsilon_1|\mathsf{m}\right)=\left(\frac{1}{n}\right)^{1-1/(N_1+N_2)}\sum^n_{j=1}\mathsf{S}_{N_1+N_2}\left(\nu^\varepsilon_{1,j}|\mathsf{m}\right).$$
Due to the first step, for each $j=1,\dots,n$ there exists a midpoint $\eta^\varepsilon_j\in\mathcal{P}_\infty(M_c,\mathsf{d},\mathsf{m})$ of $\nu^\varepsilon_{0,j}$ and $\nu^\varepsilon_{1,j}$ satisfying
$$\mathsf{S}_{N_1+N_2}\left(\eta^\varepsilon_j|\mathsf{m}\right)\leq
\sigma^{(1/2)}_{K,N_1+N_2}(\theta)\mathsf{S}_{N_1+N_2}\left(\nu^\varepsilon_{0,j}|\mathsf{m}\right)+
\sigma^{(1/2)}_{K,N_1+N_2}(\theta)\mathsf{S}_{N_1+N_2}\left(\nu^\varepsilon_{1,j}|\mathsf{m}\right).$$
Since $\mathsf{M}$ is non-branching and since the measures $\nu^\varepsilon_{0,j}$ for $j=1,\dots,n$ are mutually singular, also the $\eta^\varepsilon_j$ are mutually singular for $j=1,\dots,n$ -- we refer to Lemma \ref{midmutsing}. Therefore,
$$\eta^\varepsilon:=\frac{1}{n}\sum^n_{j=1}\eta^\varepsilon_j$$
satisfies
$$\mathsf{S}_{N_1+N_2}\left(\eta^\varepsilon|\mathsf{m}\right)=\left(\frac{1}{n}\right)^{1-1/(N_1+N_2)}\sum^n_{j=1}\mathsf{S}_{N_1+N_2}\left(\eta^\varepsilon_j|\mathsf{m}\right)$$
and consequently,
\begin{align*}
\mathsf{S}_{N_1+N_2}\left(\eta^\varepsilon|\mathsf{m}\right)
&\leq\sigma^{(1/2)}_{K,N_1+N_2}(\theta)\mathsf{S}_{N_1+N_2}\left(\nu^\varepsilon_0|\mathsf{m}\right)+
\sigma^{(1/2)}_{K,N_1+N_2}(\theta)\mathsf{S}_{N_1+N_2}\left(\nu^\varepsilon_1|\mathsf{m}\right)\\
&\leq\sigma^{(1/2)}_{K,N_1+N_2}(\theta)\mathsf{S}_{N_1+N_2}\left(\nu_0|\mathsf{m}\right)+
\sigma^{(1/2)}_{K,N_1+N_2}(\theta)\mathsf{S}_{N_1+N_2}\left(\nu_1|\mathsf{m}\right)+2\varepsilon.
\end{align*}
Moreover, $\eta^\varepsilon$ is an approximate midpoint of $\nu_0$ and $\nu_1$,
\begin{align*}
\mathsf{d_W}\left(\nu_0,\eta^\varepsilon\right)
\leq\mathsf{d_W}\left(\nu^\varepsilon_0,\eta^\varepsilon\right)+\varepsilon
&\leq\left[\frac{1}{n}\sum^n_{j=1}\mathsf{d}^2_\mathsf{W}\left(\nu^\varepsilon_{0,j},\eta^\varepsilon_j\right)\right]^{1/2}+\varepsilon\\
&\leq\frac{1}{2}\mathsf{d_W}\left(\nu^\varepsilon_0,\nu^\varepsilon_1\right)+2\varepsilon
\leq\frac{1}{2}\mathsf{d_W}\left(\nu_0,\nu_1\right)+3\varepsilon,
\end{align*}
a similar calculation holds true for $\mathsf{d_W}\left(\eta^\varepsilon,\nu_1\right)$. According to the compactness of $(M_c,\mathsf{d})$, the family $\left\{\eta^\varepsilon:\varepsilon>0\right\}$ of approximate midpoints is tight. Hence, there exists a suitable subsequence $\left(\eta^{\varepsilon_k}\right)_{k\in\N}$ converging to some $\eta\in\mathcal{P}_\infty(M_c,\mathsf{d},\mathsf{m})$. Continuity of the Wasserstein distance $\mathsf{d_W}$ and lower semi-continuity of the R\'enyi entropy functional $\mathsf{S}_{N_1+N_2}(\cdot|\mathsf{m})$ imply that $\eta$ is a midpoint of $\nu_0$ and $\nu_1$ and that
$$\mathsf{S}_{N_1+N_2}\left(\eta|\mathsf{m}\right)
\leq\sigma^{(1/2)}_{K,N_1+N_2}(\theta)\mathsf{S}_{N_1+N_2}\left(\nu_0|\mathsf{m}\right)+
\sigma^{(1/2)}_{K,N_1+N_2}(\theta)\mathsf{S}_{N_1+N_2}\left(\nu_1|\mathsf{m}\right).$$
Applying Proposition \ref{midpoints} finally yields the claim.
\end{proof}

\section{From Local to Global}

\begin{theorem}[$\mathsf{CD}^*_{\mathsf{loc}}(K,N)$ $\Leftrightarrow$ $\mathsf{CD}^*(K,N)$] \label{glob}
Let $K,N\in\R$ with $N\geq 1$ and let $(\mathsf{M},\mathsf{d},\mathsf{m})$ be a non-branching metric measure space. We assume additionally that $\mathcal{P}_\infty(\mathsf{M},\mathsf{d},\mathsf{m})$ is a geodesic space. Then $(\mathsf{M},\mathsf{d},\mathsf{m})$ satisfies $\mathsf{CD}^*(K,N)$ globally if and only if it satisfies $\mathsf{CD}^*(K,N)$ locally.
\end{theorem}

\begin{proof}
Note that in any case, $\mathsf{supp}(\m)$ will be proper: The fact that $\mathcal{P}_\infty (\mathsf{M},\mathsf{d},\mathsf{m})$ is a geodesic space implies that $\mathsf{supp}(\m)$ is a length space. Combined with its local compactness due to Remark \ref{remark}(iv), this yields the properness of $\mathsf{supp}(\m)$.

\bigskip

We confine ourselves to treating the case $K>0$. The general one follows by analogous calculations.

For each number $k\in\N\cup\{0\}$ we define a set $I_k$ of points in time,
$$I_k:=\{l2^{-k}: l=0,\dots,2^k\}.$$

For a given geodesic $\Gamma:[0,1]\rightarrow\mathcal{P}_\infty(\mathsf{M},\mathsf{d},\mathsf{m})$ we denote by $\mathcal{G}^\Gamma_k$ the set of all geodesics $[x]:=(x_t)_{0\leq t\leq 1}$ in $\mathsf{M}$ satisfying $x_t\in\mathsf{supp}(\Gamma(t))=:\mathcal{S}_t$ for all $t\in I_k$.

We consider $o\in\mathsf{supp(\mathsf{m})}$ and $R>0$ and set $M_b:=B_R(o)\cap\mathsf{supp(\mathsf{m})}$ as well as $M_c:=\overline{B_{2R}(o)}\cap\mathsf{supp(\mathsf{m})}$. Now, we formulate a property $\mathsf{C(k)}$ for every $k\in\N\cup\{0\}$:

\bigskip

$\mathsf{C(k)}$: For each geodesic $\Gamma:[0,1]\rightarrow\mathcal{P}_\infty(\mathsf{M},\mathsf{d},\mathsf{m})$ satisfying
$\Gamma(0),\Gamma(1)\in\mathcal{P}_\infty(M_b,\mathsf{d},\mathsf{m})$ and for each pair $s,t\in I_k$ with $t-s=2^{-k}$
there exists a midpoint $\eta(s,t)\in\mathcal{P}_\infty(\mathsf{M},\mathsf{d},\mathsf{m})$ of $\Gamma(s)$ and $\Gamma(t)$ such that
$$\mathsf{S}_{N'}(\eta(s,t)|\mathsf{m})\leq\sigma^{(1/2)}_{K,N'}(\theta_{s,t})\mathsf{S}_{N'}(\Gamma(s)|\mathsf{m})+
\sigma^{(1/2)}_{K,N'}(\theta_{s,t})\mathsf{S}_{N'}(\Gamma(t)|\mathsf{m}),$$
for all $N'\geq N$ where
$$\theta_{s,t}:=\inf_{[x]\in\mathcal{G}^\Gamma_k}\mathsf{d}(x_s,x_t).$$

Our first claim is:
\begin{claim}
For each $k\in\N$, $\mathsf{C(k)}$ implies $\mathsf{C(k-1)}$.
\end{claim}

In order to prove this claim, let $k\in\N$ with property $\mathsf{C(k)}$ be given. Moreover, let a geodesic $\Gamma$ in $\mathcal{P}_\infty(\mathsf{M},\mathsf{d},\mathsf{m})$ satisfying $\Gamma(0),\Gamma(1)\in\mathcal{P}_\infty(M_b,\mathsf{d},\mathsf{m})$ and numbers $s,t\in I_{k-1}$ with $t-s=2^{1-k}$ be given. We put $\theta:=\inf_{[x]\in\mathcal{G}^\Gamma_{k-1}}\mathsf{d}(x_s,x_t)$, and we define iteratively a sequence
$(\Gamma^{(i)})_{i\in\N\cup\{0\}}$ of geodesics in $\mathcal{P}_\infty(M_c,\mathsf{d},\mathsf{m})$ coinciding with $\Gamma$ on $[0,s]\cup [t,1]$ as follows:

Start with $\Gamma^{(0)}:=\Gamma$.
Assuming that $\Gamma^{(2i)}$ is already given, let $\Gamma^{(2i+1)}$ be any geodesic in $\mathcal{P}_\infty(M_c,\mathsf{d},\mathsf{m})$ which coincides with $\Gamma$ on $[0,s]\cup [t,1]$, for which
$\Gamma^{(2i+1)}\left(s+2^{-(k+1)}\right)$ is a midpoint of $\Gamma(s)=\Gamma^{(2i)}(s)$ and $\Gamma^{(2i)}\left(s+2^{-k}\right)$ and for which
$\Gamma^{(2i+1)}\left(s+3\cdot2^{-(k+1)}\right)$ is a midpoint of $\Gamma^{(2i)}\left(s+2^{-k}\right)$ and $\Gamma(t)=\Gamma^{(2i)}(t)$ satisfying
\begin{align*}
\mathsf{S}_{N'}&\left(\Gamma^{(2i+1)}\left(s+2^{-(k+1)}\right)|\mathsf{m}\right)\leq\\
&\leq\sigma^{(1/2)}_{K,N'}\left(\theta^{(2i+1)}\right)\mathsf{S}_{N'}(\Gamma(s)|\mathsf{m})+
\sigma^{(1/2)}_{K,N'}\left(\theta^{(2i+1)}\right)\mathsf{S}_{N'}\left(\Gamma^{(2i)}\left(s+2^{-k}\right)|\mathsf{m}\right)
\end{align*}
for all $N'\geq N$ where
$$\theta^{(2i+1)}:=\inf_{[x]\in\mathcal{G}^{\Gamma^{(2i)}}_k}\mathsf{d}\left(x_s,x_{s+2^{-k}}\right)\geq\tfrac{1}{2}\theta,$$
that is,
\begin{align*}
\mathsf{S}_{N'}&\left(\Gamma^{(2i+1)}\left(s+2^{-(k+1)}\right)|\mathsf{m}\right)\leq\\
&\hspace{2cm}\leq\sigma^{(1/2)}_{K,N'}\left(\tfrac{1}{2}\theta\right)\mathsf{S}_{N'}(\Gamma(s)|\mathsf{m})+
\sigma^{(1/2)}_{K,N'}\left(\tfrac{1}{2}\theta\right)\mathsf{S}_{N'}\left(\Gamma^{(2i)}\left(s+2^{-k}\right)|\mathsf{m}\right)
\end{align*}
for all $N'\geq N$ and accordingly,
\begin{align*}
\mathsf{S}_{N'}&\left(\Gamma^{(2i+1)}\left(s+3\cdot2^{-(k+1)}\right)|\mathsf{m}\right)\leq\\
&\hspace{2cm}\leq\sigma^{(1/2)}_{K,N'}\left(\tfrac{1}{2}\theta\right)\mathsf{S}_{N'}\left(\Gamma^{(2i)}\left(s+2^{-k}\right)|\mathsf{m}\right)
+\sigma^{(1/2)}_{K,N'}\left(\tfrac{1}{2}\theta\right)\mathsf{S}_{N'}(\Gamma(t)|\mathsf{m})
\end{align*}
for all $N'\geq N$. Such midpoints exist due to $\mathsf{C(k)}$.

Now let $\Gamma^{(2i+2)}$ be any geodesic
in $\mathcal{P}_\infty(M_c,\mathsf{d},\mathsf{m})$ which coincides with $\Gamma^{(2i+1)}$ on $[0,s+2^{-(k+1)}]\cup [s+3\cdot2^{-(k+1)},1]$ and
for which $\Gamma^{(2i+2)}\left(s+2^{-k}\right)$ is a midpoint of $\Gamma^{(2i+1)}\left(s+2^{-(k+1)}\right)$
and $\Gamma^{(2i+1)}\left(s+3\cdot2^{-(k+1)}\right)$ satisfying
\begin{align*}
\mathsf{S}_{N'}&\left(\Gamma^{(2i+2)}\left(s+2^{-k}\right)|\mathsf{m}\right)\leq\\
&\leq\sigma^{(1/2)}_{K,N'}\left(\tfrac{1}{2}\theta\right)\mathsf{S}_{N'}\left(\Gamma^{(2i+1)}\left(s+2^{-(k+1)}\right)|\mathsf{m}\right)
+\\
&\hspace{4cm}+\sigma^{(1/2)}_{K,N'}\left(\tfrac{1}{2}\theta\right)\mathsf{S}_{N'}\left(\Gamma^{(2i+1)}\left(s+3\cdot2^{-(k+1)}\right)|\mathsf{m}\right)
\end{align*}
for all $N'\geq N$. Again such a midpoint exists according to $\mathsf{C(k)}$. This yields a sequence $(\Gamma^{(i)})_{i\in\N\cup\{0\}}$ of geodesics. Combining the above inequalities yields
\begin{align*}
\mathsf{S}_{N'}&\left(\Gamma^{(2i+2)}\left(s+2^{-k}\right)|\mathsf{m}\right)\leq\\
&\leq2\sigma^{(1/2)}_{K,N'}\left(\tfrac{1}{2}\theta\right)^2\mathsf{S}_{N'}\left(\Gamma^{(2i)}\left(s+2^{-k}\right)|\mathsf{m}\right)+\\
&\hspace{4cm}+\sigma^{(1/2)}_{K,N'}\left(\tfrac{1}{2}\theta\right)^2\mathsf{S}_{N'}\left(\Gamma(s)|\mathsf{m}\right)
+\sigma^{(1/2)}_{K,N'}\left(\tfrac{1}{2}\theta\right)^2\mathsf{S}_{N'}\left(\Gamma(t)|\mathsf{m}\right)
\end{align*}
and by iteration,
\begin{align*}
\mathsf{S}_{N'}&\left(\Gamma^{(2i)}\left(s+2^{-k}\right)|\mathsf{m}\right)\leq\\
&\leq 2^i\sigma^{(1/2)}_{K,N'}\left(\tfrac{1}{2}\theta\right)^{2i}\mathsf{S}_{N'}\left(\Gamma\left(s+2^{-k}\right)|\mathsf{m}\right)+\\
&\hspace{4cm}+\tfrac{1}{2}\sum^i_{k=1}\left(2\sigma^{(1/2)}_{K,N'}\left(\tfrac{1}{2}\theta\right)^2\right)^k\left[\mathsf{S}_{N'}\left(\Gamma(s)|\mathsf{m}\right)
+\mathsf{S}_{N'}\left(\Gamma(t)|\mathsf{m}\right)\right]
\end{align*}
for all $N'\geq N$.

By compactness of $\mathcal{P}(M_c,\mathsf{d})$, there exists a suitable subsequence of $\left(\Gamma^{(2i)}\left(s+2^{-k}\right)\right)_{i\in\N\cup\{0\}}$
converging to some $\eta\in\mathcal{P}(M_c,\mathsf{d})$. Continuity of the distance implies that $\eta$ is a midpoint of $\Gamma(s)$ and $\Gamma(t)$
and the lower semi-continuity of the R\'enyi entropy functional implies
$$\mathsf{S}_{N'}(\eta|\mathsf{m})\leq\sigma^{(1/2)}_{K,N'}(\theta)\mathsf{S}_{N'}(\Gamma(s)|\mathsf{m})+\sigma^{(1/2)}_{K,N'}(\theta)\mathsf{S}_{N'}(\Gamma(t)|\mathsf{m})$$
for all $N'\geq N$. This proves property $\mathsf{C(k-1)}$. At this point, we do not want to suppress the calculations leading to this last implication: For all $N'\geq N$, we have
\begin{align*}
\sigma^{(1/2)}_{K,N'}\left(\tfrac{1}{2}\theta\right)&=\frac{\sin\left(\tfrac{1}{4}\theta\sqrt{K/N'}\right)}{\sin\left(\tfrac{1}{2}\theta\sqrt{K/N'}\right)}
=\frac{\sin\left(\tfrac{1}{4}\theta\sqrt{K/N'}\right)}{2\sin\left(\tfrac{1}{4}\theta\sqrt{K/N'}\right)\cos\left(\tfrac{1}{4}\theta\sqrt{K/N'}\right)}\\
&=\frac{1}{2\cos\left(\tfrac{1}{4}\theta\sqrt{K/N'}\right)}.
\end{align*}
In the case $2\sigma^{(1/2)}_{K,N'}\left(\tfrac{1}{2}\theta\right)^2<1$,
\begin{align*}
\tfrac{1}{2}\lim_{i\to\infty}\sum^i_{k=1}\left(2\sigma^{(1/2)}_{K,N'}\left(\tfrac{1}{2}\theta\right)^2\right)^k&
=\tfrac{1}{2}\left[\left(1-2\sigma^{(1/2)}_{K,N'}\left(\tfrac{1}{2}\theta\right)^2\right)^{-1}-1\right]\\
&=\tfrac{1}{2}\left[\left(\frac{2\cos^2\left(\tfrac{1}{4}\theta\sqrt{K/N'}\right)-1}{2\cos^2\left(\tfrac{1}{4}\theta\sqrt{K/N'}\right)}\right)^{-1}-1\right]\\
&=\tfrac{1}{2}\left[\frac{2\cos^2\left(\tfrac{1}{4}\theta\sqrt{K/N'}\right)}{\cos\left(\tfrac{1}{2}\theta\sqrt{K/N'}\right)}-1\right]\\
&=\tfrac{1}{2}\left[\frac{\cos\left(\tfrac{1}{2}\theta\sqrt{K/N'}\right)+1-\cos\left(\tfrac{1}{2}\theta\sqrt{K/N'}\right)}
{\cos\left(\tfrac{1}{2}\theta\sqrt{K/N'}\right)}\right]\\
&=\frac{1}{2\cos\left(\tfrac{1}{2}\theta\sqrt{K/N'}\right)}=\sigma^{(1/2)}_{K,N'}\left(\theta\right).
\end{align*}

The case $2\sigma^{(1/2)}_{K,N'}\left(\tfrac{1}{2}\theta\right)^2\geq1$ is trivial since then $\sigma^{(1/2)}_{K,N'}\left(\theta\right)=\infty$ by convention.

\bigskip

According to our curvature assumption, each point $x\in \mathsf{M}$ has a neighborhood $M(x)$ such that probability measures in $\mathcal{P}_\infty(\mathsf{M},\mathsf{d},\mathsf{m})$ which are supported in $M(x)$ can be joined by a geodesic in $\mathcal{P}_\infty(\mathsf{M},\mathsf{d},\mathsf{m})$ satisfying (\ref{ren}). By compactness of $M_c$, there exist $\lambda>0$, $n\in\N$, finitely many disjoint sets $L_1,L_2,\dots,L_n$ covering $M_c$, and closed sets $M_j\supseteq B_{\lambda}(L_j)$ for $j=1,\dots,n$, such that probability measures in $\mathcal{P}_\infty(M_j,\mathsf{d},\mathsf{m})$ can be joined by geodesics in $\mathcal{P}_\infty(\mathsf{M},\mathsf{d},\mathsf{m})$ satisfying (\ref{ren}). Choose $\kappa\in\N$ such that
$$2^{-\kappa}\mathsf{diam}(M_c,\mathsf{d},\mathsf{m})\leq\lambda.$$
Our next claim is:
\begin{claim}
Property $\mathsf{C(\kappa)}$ is satisfied.
\end{claim}
In order to prove this claim, we consider a geodesic $\Gamma$ in $\mathcal{P}_\infty(\mathsf{M},\mathsf{d},\mathsf{m})$ satisfying $\Gamma(0),\Gamma(1)\in\mathcal{P}_\infty(M_b,\mathsf{d},\mathsf{m})$ and numbers $s,t\in I_{\kappa}$ with $t-s=2^{-\kappa}$. Let $\mathsf{\hat q}$ be a coupling of $\Gamma(l2^{-\kappa})$ for $l=0,\dots,2^{\kappa}$ on $\mathsf{M}^{2^\kappa+1}$ such that for $\mathsf{\hat q}$-almost every $(x_l)_{l=0,\dots,2^{\kappa}}\in \mathsf{M}^{2^\kappa+1}$ the points $x_s,x_t$ lie on some geodesic connecting $x_0$ and $x_1$ with
\begin{equation} \label{lambda}
\mathsf{d}(x_s,x_t)=|t-s|\mathsf{d}(x_0,x_1)\leq 2^{-\kappa}\mathsf{diam}(M_c,\mathsf{d},\mathsf{m})\leq\lambda.
\end{equation}
Define probability measures $\Gamma_j(s)$ and $\Gamma_j(t)$ for $j=1,\dots,n$ by
$$\Gamma_j(s)(A):=\frac{1}{\alpha_j}\Gamma(s)(A\cap L_j)
=\frac{1}{\alpha_j}\mathsf{\hat q}(\underset{\textrm{$(2^\kappa+1)$ factors}}{\underbrace{\mathsf{M}\times\dots\times(\underset{\textrm{$(2^\kappa s+1)$-th factor}}{\underset{\uparrow}{A}}\cap L_j)\times \mathsf{M}\times\dots\times \mathsf{M})}}$$
and
$$\Gamma_j(t)(A):=\frac{1}{\alpha_j}\mathsf{\hat q}(\mathsf{M}\times\dots\times L_j\times \underset{\textrm{$(2^\kappa t+1)$-th factor}}{\underset{\uparrow}{A}}\times\dots\times \mathsf{M})$$
provided that $\alpha_j:=\Gamma_s(L_j)\not=0$. Otherwise, define $\Gamma_j(s)$ and $\Gamma_j(t)$ arbitrarily.
Then $\mathsf{supp}(\Gamma_j(s))\subseteq\overline{L_j}$ which combined with inequality (\ref{lambda}) implies
$$\mathsf{supp}(\Gamma_j(s))\cup\mathsf{supp}(\Gamma_j(t))\subseteq\overline{B_\lambda(L_j)}\subseteq M_j.$$
Therefore, for each $j\in\{1,\dots,n\}$, the assumption ``$(\mathsf{M},\mathsf{d},\mathsf{m})$ satisfies $\mathsf{CD}^*(K,N)$ locally'' can be applied to the probability measures $\Gamma_j(s)$ and $\Gamma_j(t)\in\mathcal{P}_\infty(M_j,\mathsf{d},\mathsf{m})$. It yields the existence of a midpoint $\eta_j(s,t)$ of $\Gamma_j(s)$ and $\Gamma_j(t)$ with the property that
\begin{equation} \label{etaj}
\mathsf{S}_{N'}(\eta_j(s,t)|\mathsf{m})\leq\sigma^{(1/2)}_{K,N'}(\theta_{s,t})\mathsf{S}_{N'}(\Gamma_j(s)|\mathsf{m})+
 \sigma^{(1/2)}_{K,N'}(\theta_{s,t})\mathsf{S}_{N'}(\Gamma_j(t)|\mathsf{m})
\end{equation}
for all $N'\geq N$ where
$$\theta_{s,t}:=\inf_{[x]\in\mathcal{G}^\Gamma_{\kappa}}\mathsf{d}(x_s,x_t).$$
Define
$$\eta(s,t):=\sum^n_{j=1}\alpha_j\eta_j(s,t).$$
Then, $\eta(s,t)$ is a midpoint of $\Gamma(s)=\sum^n_{j=1}\alpha_j\Gamma_j(s)$ and
$\Gamma(t)=\sum^n_{j=1}\alpha_j\Gamma_j(t)$. Moreover, since the $\Gamma_j(s)$ are mutually singular for $j=1,\dots,n$ and since $\mathsf{M}$ is non-branching, also the $\eta_j(s,t)$ are mutually singular for $j=1,\dots,n$ due to Lemma \ref{midmutsing}. Therefore, for all $N'\geq N$,
\begin{equation} \label{eta}
\mathsf{S}_{N'}(\eta(s,t)|\mathsf{m})=\sum^n_{j=1}\alpha^{1-1/N'}_j\mathsf{S}_{N'}(\eta_j(s,t)|\mathsf{m})
\end{equation}
and
\begin{equation} \label{gammas}
\mathsf{S}_{N'}(\Gamma(s)|\mathsf{m})=\sum^n_{j=1}\alpha^{1-1/N'}_j\mathsf{S}_{N'}(\Gamma_j(s)|\mathsf{m}),
\end{equation}
whereas
\begin{equation} \label{gammat}
\mathsf{S}_{N'}(\Gamma(t)|\mathsf{m})\geq\sum^n_{j=1}\alpha^{1-1/N'}_j\mathsf{S}_{N'}(\Gamma_j(t)|\mathsf{m}),
\end{equation}
since the $\Gamma_j(t)$ are not necessarily mutually singular for $j=1,\dots,n$. Summing up (\ref{etaj}) for $j=1,\dots,n$
and using (\ref{eta})--(\ref{gammat}) yields
$$\mathsf{S}_{N'}(\eta(s,t)|\mathsf{m})\leq\sigma^{(1/2)}_{K,N'}(\theta_{s,t})\mathsf{S}_{N'}(\Gamma(s)|\mathsf{m})+
 \sigma^{(1/2)}_{K,N'}(\theta_{s,t})\mathsf{S}_{N'}(\Gamma(t)|\mathsf{m})$$
for all $N'\geq N$. This proves property $\mathsf{C(\kappa)}$.

In order to finish the proof let two probability measures $\nu_0,\nu_1\in\mathcal{P}_\infty(M_b,\mathsf{d},\mathsf{m})$ be given. By assumption there exists a geodesic $\Gamma$ in $\mathcal{P}_\infty(\mathsf{M},\mathsf{d},\mathsf{m})$ connecting them. According to our second claim, property $\mathsf{C(\kappa)}$ is satisfied and according to our first claim, this implies $\mathsf{C(k)}$ for all $k=\kappa-1,\kappa-2,\dots,0$. Property $\mathsf{C(0)}$ finally states that there exists a midpoint $\eta\in\mathcal{P}_\infty(\mathsf{M},\mathsf{d},\mathsf{m})$ of $\Gamma(0)=\nu_0$ and $\Gamma(1)=\nu_1$ with
$$\mathsf{S}_{N'}(\eta|\mathsf{m})\leq\sigma^{(1/2)}_{K,N'}(\theta)\mathsf{S}_{N'}(\Gamma(0)|\mathsf{m})+\sigma^{(1/2)}_{K,N'}(\theta)\mathsf{S}_{N'}(\Gamma(1)|\mathsf{m}),$$
for all $N'\geq N$ where
$$\theta:=\inf_{x_0\in\mathcal{S}_0,x_1\in\mathcal{S}_1}\mathsf{d}(x_0,x_1).$$
This proves Theorem \ref{glob}.
\end{proof}

\begin{corollary}[$\mathsf{CD}^*_{\mathsf{loc}}(K-,N)\Leftrightarrow\mathsf{CD}^*(K,N)$]
Fix two numbers $K,N\in\R$. A non-branching metric measure space $(\mathsf{M},\mathsf{d},\mathsf{m})$ fulfills the reduced curvature-dimension condition $\mathsf{CD}^*(K',N)$ locally for all $K'< K$ if and only if it satisfies the condition $\mathsf{CD}^*(K,N)$ globally.
\end{corollary}

\begin{proof}
Given any $K'<K$, the condition $\mathsf{CD}^*(K',N)$ is deduced from $\mathsf{CD}^*_{\mathsf{loc}}(K',N)$ according to the above localization theorem. Due to the stability of the reduced curvature-dimension condition stated in Theorem \ref{stab}, $\mathsf{CD}^*(K',N)$ for all $K'<K$ implies $\mathsf{CD}^*(K,N)$.
\end{proof}

\begin{proposition}[$\mathsf{CD}^*_{\mathsf{loc}}(K-,N)\Leftrightarrow\mathsf{CD}_{\mathsf{loc}}(K-,N)$] \label{forall}
Fix two numbers $K,N\in\R$. A metric measure space $(\mathsf{M},\mathsf{d},\mathsf{m})$ fulfills the reduced curvature-dimension condition $\mathsf{CD}^*(K',N)$ locally for all $K'< K$ if and only if it satisfies the original condition $\mathsf{CD}(K',N)$ locally for all $K'< K$.
\end{proposition}

\begin{proof}
As remarked in the past, we content ourselves with the case $K>0$. Again, the general one can be deduced from analogous calculations. The implication \textquotedblleft$\mathsf{CD}^*_{\mathsf{loc}}(K-,N)\Leftarrow\mathsf{CD}_{\mathsf{loc}}(K-,N)$\textquotedblright \ follows from analogous arguments leading to part (i) of Proposition \ref{implications}.

The implication \textquotedblleft$\mathsf{CD}^*_{\mathsf{loc}}(K-,N)\Rightarrow\mathsf{CD}_{\mathsf{loc}}(K-,N)$\textquotedblright \ is based on the fact that the coefficients $\tau^{(t)}_{K,N}(\theta)$ and $\sigma^{(t)}_{K,N}(\theta)$ are ``almost identical'' for $\theta\ll 1$: In order to be precise, we consider $0<K'<\tilde K<K$ and $\theta\ll 1$ and compare the relevant coefficients $\tau^{(t)}_{K',N}(\theta)$ and $\sigma^{(t)}_{\tilde K,N}(\theta)$:
\begin{align*}
\left[\tau^{(t)}_{K',N}(\theta)\right]^N&=t\left[\frac{\sin\left(t\theta\sqrt{\tfrac{K'}{N-1}}\right)}{\sin\left(\theta\sqrt{\tfrac{K'}{N-1}}\right)}\right]^{N-1}\\
&=t^N\left[\frac{1-\tfrac{1}{6}t^2\theta^2\tfrac{K'}{N-1}+O(\theta^4)}{1-\tfrac{1}{6}\theta^2\tfrac{K'}{N-1}+O(\theta^4)}\right]^{N-1}\\
&=t^N\left[1+\tfrac{1}{6}(1-t^2)\theta^2\tfrac{K'}{N-1}+O(\theta^4)\right]^{N-1}\\
&=t^N\left[1+\tfrac{1}{6}(1-t^2)\theta^2K'+O(\theta^4)\right].
\end{align*}
And accordingly,
\begin{align*}
\left[\sigma^{(t)}_{\tilde K,N}(\theta)\right]^N&=\left[\frac{\sin\left(t\theta\sqrt{\tfrac{\tilde K}{N}}\right)}{\sin\left(\theta\sqrt{\tfrac{\tilde K}{N}}\right)}\right]^N\\
&= t^N\left[\frac{1-\tfrac{1}{6}t^2\theta^2\tfrac{\tilde K}{N}+O(\theta^4)}{1-\tfrac{1}{6}\theta^2\tfrac{\tilde K}{N}+O(\theta^4)}\right]^N\\
&=t^N\left[1+\tfrac{1}{6}(1-t^2)\theta^2\tfrac{\tilde K}{N}+O(\theta^4)\right]^N\\
&=t^N\left[1+\tfrac{1}{6}(1-t^2)\theta^2\tilde K+O(\theta^4)\right].
\end{align*}
Now we choose $\theta^\ast>0$ in such a way that
$$\tau^{(t)}_{K',N}(\theta)\leq\sigma^{(t)}_{\tilde K,N}(\theta)$$
for all $0\leq\theta\leq\theta^\ast$ and all $t\in[0,1]$. According to our curvature assumption, each point $x\in\mathsf{M}$ has a neighborhood $M(x)\subseteq\mathsf{M}$ such that every two probability measures $\nu_0,\nu_1\in\mathcal{P}_\infty(M(x),\mathsf{d},\mathsf{m})$ can be joined by a geodesic in $\mathcal{P}_\infty(\mathsf{M},\mathsf{d},\mathsf{m})$ satisfying (\ref{ren}). In order to prove that $(\mathsf{M},\mathsf{d},\mathsf{m})$ satisfies $\mathsf{CD}(K',N)$ locally, we set for $x\in\M$,
$$M'(x):=M(x)\cap B_{\theta^\ast}(x)$$
and consider $\nu_0,\nu_1\in\mathcal{P}_\infty(M'(x),\mathsf{d},\mathsf{m})$. As indicated above, due to $\mathsf{CD}^*_{\mathsf{loc}}(\tilde K,N)$ there exist an optimal
coupling $\mathsf{q}$ of $\nu_0=\rho_0\mathsf{m}$ and $\nu_1=\rho_1\mathsf{m}$ and a geodesic
$\Gamma:[0,1]\rightarrow\mathcal{P}_\infty(\mathsf{M},\mathsf{d},\mathsf{m})$ connecting $\nu_0$ and $\nu_1$
such that
\begin{align*}
\mathsf{S}_{N'}&(\Gamma(t)|\mathsf{m})\leq\\
&\leq-\int_{\mathsf{M}\times
\mathsf{M}}\left[\sigma^{(1-t)}_{\tilde K,N'}(\underset{\leq\theta^\ast}{\underbrace{\mathsf{d}(x_0,x_1)}})\rho^{-1/N'}_0(x_0)+
\sigma^{(t)}_{\tilde K,N'}(\underset{\leq\theta^\ast}{\underbrace{\mathsf{d}(x_0,x_1)}})\rho^{-1/N'}_1(x_1)\right]d\mathsf{q}(x_0,x_1)\\
&\leq-\int_{\mathsf{M}\times
\mathsf{M}}\left[\tau^{(1-t)}_{K',N'}(\mathsf{d}(x_0,x_1))\rho^{-1/N'}_0(x_0)+
\tau^{(t)}_{K',N'}(\mathsf{d}(x_0,x_1))\rho^{-1/N'}_1(x_1)\right]d\mathsf{q}(x_0,x_1)
\end{align*}
for all $t\in [0,1]$ and all $N'\geq N$.
\end{proof}

\begin{remark}
The proofs of Theorem \ref{tensorization} and Theorem \ref{glob}, respectively, do not extend to the ori\-ginal curvature-dimension condition $\mathsf{CD}(K,N)$. An immediate obstacle is that no analogous statements of rather technical tools like Lemma \ref{sigma} and Proposition \ref{midpoints} are known due to the more complicated nature of the coefficients $\tau^{(t)}_{K,N}(\cdot)$. It is still an open question whether $\mathsf{CD}(K,N)$ satisfies the tensorization or the local-to-global property.
\end{remark}




\section{Geometric and Functional Analytic Consequences}

\subsection{Geometric Results}

The weak versions of the geometric statements derived from $\mathsf{CD}(K,N)$ in \cite{sb} follow by using analogous arguments replacing the coefficients $\tau^{(t)}_{K,N}(\cdot)$ by $\sigma^{(t)}_{K,N}(\cdot)$.

Note that we do not use the assumption of non-branching metric measure spaces in this whole section and that Corollary \ref{bonnet} and Theorem \ref{lichnerowicz} follow immediately from the strong versions in \cite{sb} in combination with Proposition \ref{implications}(ii).

\begin{proposition}[Generalized Brunn-Minkowski inequality] \label{brunn}
Assume that $(\mathsf{M},\mathsf{d},\mathsf{m})$ satisfies the condition $\mathsf{CD}^*(K,N)$ for two real parameters $K,N$ with $N\geq 1$. Then for all measurable sets $A_0,A_1\subseteq \mathsf{M}$ with $\mathsf{m}(A_0),\mathsf{m}(A_1)>0$ and all $t\in[0,1]$,
\begin{equation} \label{gbm}
\mathsf{m}(A_t)\geq
\sigma^{(1-t)}_{K,N}(\Theta)\cdot
\mathsf{m}(A_0)^{1/N}+\sigma^{(t)}_{K,N}(\Theta)\cdot\mathsf{m}(A_1)^{1/N}
\end{equation}
where $A_t$ denotes the set of points which divide geodesics starting in $A_0$ and ending in $A_1$ with ratio $t:(1-t)$ and where $\Theta$ denotes the minimal/maximal length of such geodesics
\begin{equation*}
\Theta:=
\begin{cases}
\underset{x_0\in A_0,x_1\in A_1}{\inf}\mathsf{d}(x_0,x_1),& K\geq 0\\
\underset{x_0\in A_0,x_1\in A_1}{\sup}\mathsf{d}(x_0,x_1),& K<0.
\end{cases}
\end{equation*}
\end{proposition}

The Brunn-Minkowski inequality implies further geometric consequences, for example the Bishop-Gromov volume growth estimate and the Bonnet-Myers theorem.

For a fixed point $x_0\in\mathsf{supp}(\mathsf{m})$ we study the growth of the volume of closed balls centered at $x_0$ and the growth of the volume of the corresponding spheres
$$v(r):=\mathsf{m}\left(\overline{B_r(x_0)}\right) \mbox{ \ and \ } s(r):=\limsup_{\delta\to 0}\tfrac{1}{\delta}\mathsf{m}\left(\overline{B_{r+\delta}(x_0)}\setminus B_r(x_0)\right),$$
respectively.

\begin{theorem}[Generalized Bishop-Gromov volume growth inequality] \label{bishop}
Assume that the metric measure space $(\mathsf{M},\mathsf{d},\mathsf{m})$ satisfies the condition $\mathsf{CD}^*(K,N)$ for some $K,N\in\R$. Then each bounded closed set $M_{b,c}\subseteq\mathsf{supp}(\m)$ is compact and has finite volume. To be more precise, if $K>0$ then for each fixed $x_0\in\mathsf{supp}(\mathsf{m})$ and all $0<r<R\leq\pi\sqrt{N/K}$,
\begin{equation} \label{sphere}
\frac{s(r)}{s(R)}\geq\left(\frac{\sin(r\sqrt{K/N})}{\sin(R\sqrt{K/N})}\right)^N \mbox{ \ and \quad }
\frac{v(r)}{v(R)}\geq\frac{\int\limits^r_0\sin\left(t\sqrt{K/N}\right)^Ndt}{\int\limits^R_0\sin\left(t\sqrt{K/N}\right)^Ndt}.
\end{equation}
In the case $K<0$, analogous inequalities hold true (where the right-hand sides of (\ref{sphere}) are replaced by analogous expressions according to the definition of the coefficients $\sigma^{(t)}_{K,N}(\cdot)$ for negative $K$).
\end{theorem}

\begin{corollary}[Generalized Bonnet-Myers theorem] \label{bonnet}
Fix two real parameters $K>0$ and $N\geq 1$. Each metric measure space $(\mathsf{M},\mathsf{d},\mathsf{m})$ satisfying the condition $\mathsf{CD}^*(K,N)$ has compact support and its diameter $L$ has an upper bound
$$L\leq\pi\sqrt{\frac{N}{K}}.$$
\end{corollary}
\
\\
Note that in the sharp version of this estimate the factor $N$ is replaced by $N-1$.

\subsection{Lichnerowicz Estimate}

In this subsection we follow the presentation of Lott and Villani in \cite{lva}.

\begin{definition}
Given $f\in\mathsf{Lip}(\M)$, we define $|\nabla^-f|$ by
$$|\nabla^-f|(x):=\underset{y\to x}{\limsup}\frac{[f(y)-f(x)]_-}{\mathsf{d}(x,y)}$$
where for $a\in\R$, $a_-:=\max(-a,0)$.
\end{definition}

\begin{theorem}[Lichnerowicz estimate, Poincar\'e inequality] \label{lichnerowicz}
We assume that $(\M,\mathsf{d},\m)$ satisfies $\mathsf{CD}^*(K,N)$ for two real parameters $K>0$ and $N\geq 1$. Then for every $f\in\mathsf{Lip}(\M)$ fulfilling $\int_\M fd\m=0$ the following inequality holds true,
\begin{equation} \label{spectgap}
\int_\M f^2d\m\leq\frac{1}{K}\int_\M|\nabla^-f|^2d\m.
\end{equation}
\end{theorem}

\begin{remark}
In \textquoteleft regular\textquoteright \ cases, {\large $\varepsilon$}$(f,f):=\int_\M|\nabla^-f|^2 \ d\mathsf{m}$ is a quadratic form which -- by polarization -- then defines uniquely a bilinear form {\large $\varepsilon$}$(f,g)$ and a self-adjoint operator $L$ (\textquoteleft generalized Laplacian\textquoteright) through the identity {\large $\varepsilon$}$(f,g)=-\int_\M f\cdot Lg \ d\mathsf{m}$.

The inequality (\ref{spectgap}) means that $L$ admits a spectral gap $\lambda_1$ of size at least $K$,
$$\lambda_1\geq K.$$
In the sharp version, corresponding to the case where $(\M,\mathsf{d},\m)$ satisfies $\mathsf{CD}(K,N)$, the spectral gap is bounded from below by $K\frac{N}{N-1}$.
\end{remark}

\section{Universal Coverings of Metric Measure Spaces}

\subsection{Coverings and Liftings}

Let us recall some basic definitions and properties of coverings of metric (or more generally, topological) spaces. For further details we refer to \cite{bi}.

\begin{definition}[Covering]
\begin{itemize}
\item[(i)] Let $E$ and $X$ be topological spaces and $p:E\to X$ a continuous map. An open set $V\subseteq X$ is said to be evenly covered by $p$ if and only if its inverse image $p^{-1}(V)$ is a disjoint union of sets $U_i\subseteq E$ such that the restriction of $p$ to $U_i$ is a homeomorphism from $U_i$ to $V$ for each $i$ in a suitable indexset $I$. The map $p$ is a covering map (or simply covering) if and only if every point $x\in X$ has an evenly covered neighborhood. In this case, the space $X$ is called the base of the covering and $E$ the covering space.
\item[(ii)] A covering map $p:E\to X$ is called a universal covering if and only if $E$ is simply connected. In this case, $E$ is called universal covering space for $X$.
\end{itemize}
\end{definition}

The existence of a universal covering is guaranteed under some weak topological assumptions. More precisely:

\begin{theorem} \label{existence}
 If a topological space $X$ is connected, locally pathwise connected and semi-locally simply connected, then there exists a universal covering
$p:E\to X$.
\end{theorem}

For the exact meaning of the assumptions we again refer to \cite{bi}.

\begin{example}
\begin{itemize}
\item[(i)] The  map $p:\R\to\mathsf{S^1}$ given by $p(x)=(\cos(x),\sin(x))$ is a covering map.
\item[(ii)] The universal covering of the torus by the plane $P:\R^2\to\mathsf{T^2}:=\mathsf{S^1}\times\mathsf{S^1}$ is given by $P(x,y):=(p(x),p(y))$ where $p(x)=(\cos(x),\sin(x))$ is defined as in (i).
\end{itemize}
\end{example}

We consider a covering $p:E\to X$. For $x\in X$ the set $p^{-1}(x)$ is called the \textit{fiber} over $x$. This is a discrete subspace of $E$ and every $x\in X$ has a neighborhood $V$ such that $p^{-1}(V)$ is homeomorphic to $p^{-1}(x)\times V$. The disjoint subsets of $p^{-1}(V)$ mapped homeomorphically onto $V$ are called the \textit{sheets} of $p^{-1}(V)$. If $V$ is connected, the sheets of $p^{-1}(V)$ coincide with the connected components of $p^{-1}(V)$. If $E$ and $X$ are connected, the cardinality of $p^{-1}(x)$ does not depend on $x\in X$ and is called the \textit{number of sheets}. This number may be infinity.

Every covering is a local homeomorphism which implies that $E$ and $X$ have the same local topological properties.

\begin{remark}
Consider length spaces $(E,\mathsf{d}_E)$ and $(X,\mathsf{d}_X)$ and a covering map $p:E\to X$ which is additionally a local isometry. If $X$ is complete, then so is $E$.
\end{remark}

We list two essential lifting statements in topology referring to \cite{bs} for further details and the proofs.

\begin{definition}
Let $\alpha,\beta:[0,1]\rightarrow X$ be two curves in $X$ with the same end points meaning that $\alpha(0)=\beta(0)=x_0\in X$ and $\alpha(1)=\beta(1)=x_1\in X$. We say that $\alpha$ and $\beta$ are homotopic relative to $\{0,1\}$ if and only if there exists a continuous map $H:[0,1]\times [0,1]\rightarrow X$ satisfying $H(t,0)=\alpha(t)$, $H(t,1)=\beta(t)$ as well as $H(0,t)=x_0$ and $H(1,t)=x_1$ for all $t\in[0,1]$. We call $H$ a homotopy from $\alpha$ to $\beta$ relative to $\{0,1\}$.
\end{definition}

\begin{theorem}[Path lifting theorem]
Let $p:E\to X$ be a covering and let $\gamma:[0,1]\to X$ be a curve in $X$. We assume that $e_0\in E$ satisfies $p(e_0)=\gamma(0)$. Then there exists a unique curve $\alpha:[0,1]\to E$ such that $\alpha(0)=e_0$ and $p\circ\alpha=\gamma$.
\end{theorem}

\begin{theorem}[Homotopy lifting theorem]
Let $p:E\to X$ be a covering, let $\gamma_0,\gamma_1:[0,1]\to X$ be two curves in $X$ with starting point $x_0\in X$ and terminal point $x_1\in X$, and let $\alpha_0,\alpha_1:[0,1]\to E$ be the lifted curves such that $\alpha_0(0)=\alpha_1(0)$. Then every homotopy $H:[0,1]\times [0,1]\to X$ from $\gamma_0$ to $\gamma_1$ relative to $\{0,1\}$ can be lifted in a unique way to a homotopy $H':[0,1]\times [0,1]\to E$ from $\alpha_0$ to $\alpha_1$ relative to $\{0,1\}$ satisfying $H'(0,0)=\alpha_0(0)=\alpha_1(0)$.
\end{theorem}

We consider a universal covering $p:E\to X$ and distinguished points $x_0\in X$ as well as $e_0\in p^{-1}(x_0)\subseteq E$. The above lifting theorems enable us to define a function
$$\Phi:\pi_1(X,x_0)\to p^{-1}(x_0)$$
such that for $[\gamma]\in\pi_1(X,x_0)$, $\Phi([\gamma])$ is the (unique) terminal point of the lift of $\gamma$ to $E$ starting at $e_0$. Then $\Phi$ has the following property:

\begin{theorem}[Cardinality of fibers] \label{card}
The function $\Phi$ is a one-to-one correspondence of the fundamental group $\pi_1(X,x_0)$ and the fiber $p^{-1}(x_0)$.
\end{theorem}

\subsection{Lifted Metric Measure Spaces}

We consider now a non-branching metric measure space $(\mathsf{M},\mathsf{d},\mathsf{m})$ satisfying the reduced curvature-dimension condition $\mathsf{CD}^*(K,N)$ locally for two real parameters $K>0$ and $N\geq 1$ and a distinguished point $x_0\in\mathsf{M}$. Moreover, we assume that $(\mathsf{M},\mathsf{d})$ is a semi-locally simply connected length space. Then, according to Theorem \ref{existence}, there exists a universal covering $p:\mathsf{\hat M}\to\mathsf{M}$. The covering space $\mathsf{\hat M}$ inherits the length structure of the base $\mathsf{M}$ in the following way: We say that a curve $\hat\gamma$ in $\mathsf{\hat M}$ is ``admissible'' if and only if its composition with $p$ is a continuous curve in $\mathsf{M}$. The length $\mathsf{Length}(\hat\gamma)$ of an admissible curve in $\mathsf{\hat M}$ is set to the length of $p\circ\hat\gamma$ with respect to the length structure in $\mathsf{M}$. For two points $x,y\in\mathsf{\hat M}$ we define the associated distance $\mathsf{\hat d}(x,y)$ between them to be the infimum of lengths of admissible curves in $\mathsf{\hat M}$ connecting these points:
\begin{equation} \label{lift_d}
\mathsf{\hat d}(x,y):=\inf\{\mathsf{Length}(\hat\gamma)|\hat\gamma:[0,1]\to\mathsf{\hat M} \ \mathrm{admissible},\hat\gamma(0)=x, \hat\gamma(1)=y\}.
\end{equation}
Endowed with this metric, $p:(\mathsf{\hat M},\mathsf{\hat d})\to(\mathsf{M},\mathsf{d})$ is a local isometry.

Now, let $\xi$ be the family of all sets $\hat E\subseteq\mathsf{\hat M}$ such that the restriction of $p$ onto $\hat E$ is a local isometry from $\hat E$ to a measurable set $E:=p(\hat E)$ in $\mathsf{M}$. This family $\xi$ is stable under intersections, and the smallest $\sigma$-algebra $\sigma(\xi)$ containing $\xi$ is equal to the Borel-$\sigma$-algebra $\mathcal{B}(\mathsf{\hat M})$ according to the local compactness of $(\mathsf{\hat M},\mathsf{\hat d})$. We define a function $\mathsf{\hat m}:\xi\to[0,\infty[$ by $\mathsf{\hat m}(\hat E)=\mathsf{m}(p(\hat E))=\mathsf{m}(E)$ and extend it in a unique way to a measure $\mathsf{\hat m}$ on $(\mathsf{\hat M},\mathcal{B}(\mathsf{\hat M}))$.

\begin{definition}
\begin{itemize}
\item[(i)]We call the metric $\mathsf{\hat d}$ on $\mathsf{\hat M}$ defined in (\ref{lift_d}) the lift of the metric $\mathsf{d}$ on $\mathsf{M}$.
\item[(ii)]The measure $\mathsf{\hat m}$ on $(\mathsf{\hat M},\mathcal{B}(\mathsf{\hat M}))$ constructed as described above is called the lift of $\mathsf{m}$.
\item[(iii)] We call the metric measure space $(\mathsf{\hat M},\mathsf{\hat d},\mathsf{\hat m})$ the lift of $(\mathsf{M},\mathsf{d},\mathsf{m})$.
\end{itemize}
\end{definition}

\begin{theorem}[Lift]
Assume that $(\mathsf{M},\mathsf{d},\mathsf{m})$ is a non-branching metric measure space satisfying $\mathsf{CD}^*_{\mathsf{loc}}(K,N)$ for two real parameters $K>0$ and $N\geq 1$ and that $(\mathsf{M},\mathsf{d})$ is a semi-locally simply connected length space. Let $\mathsf{\hat M}$ be a universal covering space for $\mathsf{M}$ and let $(\mathsf{\hat M},\mathsf{\hat d},\mathsf{\hat m})$ be the lift of $(\mathsf{M},\mathsf{d},\mathsf{m})$. Then,
\begin{itemize}
\item[(i)]$(\mathsf{\hat M},\mathsf{\hat d},\mathsf{\hat m})$ has compact support and its diameter has an upper bound
$$\mathsf{diam}(\mathsf{\hat M},\mathsf{\hat d},\mathsf{\hat m})\leq\pi\sqrt{\frac{N}{K}}.$$
\item[(ii)]The fundamental group $\pi_1(\mathsf{M},x_0)$ of $(\mathsf{M},\mathsf{d},\mathsf{m})$ is finite.
\end{itemize}
\end{theorem}

\begin{proof}
\begin{itemize}
\item[(i)] Due to the construction of the lift, the local properties of $(\mathsf{M},\mathsf{d},\mathsf{m})$ are transferred to $(\mathsf{\hat M},\mathsf{\hat d},\mathsf{\hat m})$. That means, $(\mathsf{\hat M},\mathsf{\hat d},\mathsf{\hat m})$ is a non-branching metric measure space $(\mathsf{\hat M},\mathsf{\hat d},\mathsf{\hat m})$ satisfying $\mathsf{CD}^*(K,N)$ locally. Theorem \ref{glob} implies that $(\mathsf{\hat M},\mathsf{\hat d},\mathsf{\hat m})$ satisfies $\mathsf{CD}^*(K,N)$ globally and therefore, the diameter estimate of Bonnet-Myers -- Corollary \ref{bonnet} -- can be applied.
\item[(ii)] If the fundamental group $\pi_1(\mathsf{M},x_0)$ were infinite then the support of $\hat\m$ could not be compact according to Theorem \ref{card}.
\end{itemize}
\end{proof}

\begin{remark}
Note that there exists a universal cover for any Gromov-Hausdorff limit of a sequence of complete Riemannian manifolds with a uniform lower bound on the Ricci curvature \cite{sor}. The limit space may have infinite topological type \cite{men}.
\end{remark}


\begin{thebibliography}{999}
\bibitem[BS]{bs}
\textsc{K. Bartoszek, J. Signerska}, \textit{The Fundamental Group, Covering Spaces and Topology in Biology}. Preprint.
\bibitem[BS09]{abs}
\textsc{A. Bonciocat, K.-T. Sturm}, \textit{Mass transportation and rough curvature bounds for discrete spaces}. J. Funct. Anal. 256 (2009), no. 9, 2944--2966.
\bibitem[BBI01]{bi}
\textsc{D. Burago, Y. Burago, S. Ivanov}, \textit{A Course in Metric Geometry}. Graduate Studies in Mathematics 33 (2001), American Mathematical Society, Providence, RI.
\bibitem[CMS01]{co}
\textsc{D. Cordero-Erausquin, R. McCann, M. Schmuckenschl\"ager}, \textit{A Riemannian interpolation \`a la Borell, Brascamp and Lieb}. Invent. Math. 146 (2001), no. 2, 219--257.
\bibitem[FSS]{fss}
\textsc{S. Fang, J. Shao, K.-T. Sturm}, \textit{Wasserstein space over the Wiener space}. To appear in Probab. Theory and Relat. Fields.
\bibitem[LV07]{lva}
\textsc{J. Lott, C. Villani}, \textit{Weak curvature conditions and functional inequalities}. J. Funct. Anal. 245 (2007), no. 1, 311--333.
\bibitem[LV09]{lv}
\textsc{---} \textit{Ricci curvature for metric measure spaces via optimal transport}. Ann. of Math. (2) 169 (2009), no. 3, 903--991.
\bibitem[Men00]{men}
\textsc{X. Menguy}, \textit{Examples with bounded diameter growth and infinite topological type}. Duke Math. J. 102 (2000), no. 3, 403--412.
\bibitem[Oht07a]{oa}
\textsc{S. Ohta}, \textit{On the measure contraction property of metric measure spaces}. Comment. Math. Helv. 82 (2007), no. 4, 805--828.
\bibitem[Oht07b]{ob}
\textsc{---} \textit{Products, cones, and suspensions of spaces with the measure contraction property}. J. Lond. Math. Soc. (2) 76 (2007), no. 1, 225--236.
\bibitem[Oht]{oh}
\textsc{---} \textit{Finsler interpolation inequalities}. To appear in Calc. Var. Partial Differential Equations.
\bibitem[Oll09]{ol}
\textsc{Y. Ollivier}, \textit{Ricci curvature of Markov chains on metric spaces}. J. Funct. Anal. 256 (2009), no. 3, 810--864.
\bibitem[OV00]{ov}
\textsc{F. Otto, C. Villani}, \textit{Generalization of an inequality by Talagrand and links with the logarithmic Sobolev inequality}.
J. Funct. Anal. 173 (2000), no. 2, 361--400.
\bibitem[Pet09]{pe}
\textsc{A. Petrunin}, \textit{Alexandrov meets Lott--Villani--Sturm}. Preprint (2009).
\bibitem[RS05]{vrs}
\textsc{M.-K. von Renesse, K.-T. Sturm}, \textit{Transport inequalities, gradient estimates, entropy, and Ricci curvature}. Comm. Pure Appl. Math. 58 (2005), no. 7, 923--940.
\bibitem[SW04]{sor}
\textsc{C. Sormani, G. Wei}, \textit{Universal covers for Hausdorff limits of noncompact spaces}. Trans AMS 356 (2004), no. 3, 1233--1270.
\bibitem[Stu05]{sc}
\textsc{K.-T. Sturm}, \textit{Convex functionals of probability measures and nonlinear diffusions on manifolds}. J. Math. Pures Appl. (9) 84 (2005), no. 2, 149--168.
\bibitem[Stu06a]{sa}
\textsc{---} \textit{On the geometry of metric measure spaces. I}. Acta Math. 196 (2006), no. 1, 65--131.
\bibitem[Stu06b]{sb}
\textsc{---} \textit{On the geometry of metric measure spaces. II}. Acta Math. 196 (2006), no. 1, 133--177.
\bibitem[Vil03]{vi}
\textsc{C. Villani}, \textit{Topics in Optimal Transportation}. Graduate Studies in Mathematics 58 (2003), American Mathematical Society, Providence, RI.
\bibitem[Vil09]{vib}
\textsc{---} \textit{Optimal Transport, old and new}. Grundlehren der mathematischen Wissenschaften 338 (2009), Springer Berlin $\cdot$ Heidelberg.
\end{thebibliography}
\end{document}